\documentclass[12pt]{article}

\usepackage[margin=1in]{geometry}
\usepackage{amsmath,amsthm,amssymb}
\usepackage{enumerate}
\usepackage{xfrac}
\usepackage{accents} 
\usepackage[colorlinks=true,hidelinks]{hyperref}
    \AtBeginDocument{}
\usepackage{tikz} 
\usetikzlibrary{arrows}
\usetikzlibrary{patterns}
\usepackage{pgfplots} 
\usetikzlibrary{decorations.markings} 
\usetikzlibrary{3d,calc} 
\usepackage{pifont}
\DeclareFontFamily{OT1}{pzc}{}
  \DeclareFontShape{OT1}{pzc}{m}{it}{<-> s * [1.200] pzcmi7t}{}
  \DeclareMathAlphabet{\mathpzc}{OT1}{pzc}{m}{it}
\usepackage{fancyhdr}
\newcommand{\N}{\mathbb{N}}
\newcommand{\Z}{\mathbb{Z}}
\newcommand{\Q}{\mathbb{Q}}

\newcommand{\C}{\mathbb{C}}
\newcommand{\F}{\mathbb{F}}
\renewcommand{\L}{\mathcal{L}}
\newcommand{\E}{\mathcal{E}}
\renewcommand{\P}{\mathbb{P}}
\newcommand{\M}{\mathcal{M}}

\newcommand{\G}{\mathbb{G}}
\newcommand{\X}{\mathcal{X}}
\newcommand{\Y}{\mathcal{Y}}
\newcommand{\Zz}{\mathcal{Z}}
\DeclareFontFamily{U}{wncy}{}
    \DeclareFontShape{U}{wncy}{m}{n}{<->wncyr10}{}
    \DeclareSymbolFont{mcy}{U}{wncy}{m}{n}
    \DeclareMathSymbol{\Sha}{\mathord}{mcy}{"58} 

\newcommand{\B}{\mathcal{B}}

\newcommand{\Bun}{\ensuremath{\operatorname{Bun}}}

\renewcommand{\char}{\ensuremath{\operatorname{char}}}
\newcommand{\Spec}{\ensuremath{\operatorname{Spec}}}

\newcommand{\Aut}{\ensuremath{\operatorname{Aut}}}

\newcommand{\orb}{\mathcal{O}}

\newcommand*\cat[1]{\mathsf{#1}}

\newcommand{\Hom}{\ensuremath{\operatorname{Hom}}}

\newcommand{\sesx}[5]{0\rightarrow #1 \xrightarrow{#4} #2 \xrightarrow{#5} #3 \rightarrow 0}

\newcommand{\A}{\mathbb{A}}
\newcommand{\V}{\mathcal{V}}

\newcommand{\Div}{\ensuremath{\operatorname{Div}}}
\newcommand{\PDiv}{\ensuremath{\operatorname{PDiv}}}
\newcommand{\Pic}{\ensuremath{\operatorname{Pic}}}
\newcommand{\Cl}{\ensuremath{\operatorname{Cl}}}
\renewcommand{\div}{\ensuremath{\operatorname{div}}}

\newcommand{\et}{\ensuremath{\operatorname{\acute{e}t}}}
\newcommand{\Proj}{\ensuremath{\operatorname{Proj}}}
\newcommand{\W}{\mathbb{W}}
\newcommand{\DIV}{\mathfrak{Div}}
\newcommand{\cond}{\ensuremath{\operatorname{cond}}}

\newcommand{\h}{\frak{h}}
\newcommand{\U}{\mathcal{U}}

\renewcommand{\epsilon}{\varepsilon}

\newcommand{\ord}{\ensuremath{\operatorname{ord}}}

\newtheorem{thm}{Theorem}[section]
\newtheorem{prop}[thm]{Proposition}

\newtheorem{lem}[thm]{Lemma}
\newtheorem{defn}[thm]{Definition}
\newtheorem{cor}[thm]{Corollary}

\newtheorem{rem}[thm]{Remark}
  \let\oldrem\rem
  \renewcommand{\rem}{\oldrem\normalfont}
\newtheorem{question}{Question}
\newtheorem{ex}[thm]{Example}
  \let\oldex\ex
  \renewcommand{\ex}{\oldex\normalfont}

\begin{document}




\title{Artin--Schreier Root Stacks}
\author{Andrew Kobin\footnote{The author was partially supported by the National Science Foundation under DMS Grant No.~1900396.}}
\date{}

\maketitle



\begin{abstract}

We classify stacky curves in characteristic $p > 0$ with cyclic stabilizers of order $p$ using higher ramification data. This approach replaces the local root stack structure of a tame stacky curve, similar to the local structure of a complex orbifold curve, with a more sensitive structure called an Artin--Schreier root stack, allowing us to incorporate this ramification data directly into the stack. As an application, we compute dimensions of Riemann--Roch spaces for some examples of stacky curves in positive characteristic and suggest a program for computing spaces of modular forms in this setting. 

\end{abstract}


\section{Introduction}
\label{sec:intro}

\vspace{0.1in}
Over the complex numbers, Deligne--Mumford stacks with generically trivial stabilizer can be understood by studying their underlying complex orbifold structure. This approach leads one to the following classification of stacky curves: over $\C$, or indeed any algebraically closed field of characteristic $0$, a smooth stacky curve is uniquely determined up to isomorphism by its underlying complex curve and a finite list of numbers corresponding to the orders of the (always cyclic) stabilizer groups of the stacky points of the curve (cf. \cite{gs}). 

Such a concise statement fails in positive characteristic since stabilizer groups may be nonabelian. Even in the case of stacky curves over an algebraically closed field $k$ of characteristic $p > 0$, there exist infinitely many nonisomorphic stacky curves over $k$ with the same underlying scheme and stacky point with stabilizer group abstractly isomorphic to $\Z/p\Z$. Thus any attempt at classifying stacky curves in characteristic $p$ will require finer invariants than the order of the stabilizer group. Our main results, stated below, provide a classification of stacky curves in characteristic $p > 0$ having nontrivial stabilizers of order $p$, using higher ramification data at the stacky points. 

A new tool called an {\it Artin--Schreier root stack}, denoted $\wp_{m}^{-1}((L,s,f)/X)$, will be constructed and studied in Section~\ref{sec:ASrootstacks}, but it appears in the statements of our main results below so a brief introduction is in order. The object $\wp_{m}^{-1}((L,s,f)/X)$ is defined using the data of a line bundle $L$ on $X$ and two sections $s\in H^{0}(X,L)$ and $f\in H^{0}(X,L^{\otimes m})$. It replaces the root stack $\sqrt[r]{(L,s)/X}$ from the theory of tame stacky curves. Our main results are the following theorems. 

\begin{thm}[Theorem~\ref{thm:pcoverfactorsthroughASrootstack}]
\label{thm:1}
Suppose $\varphi : Y\rightarrow X$ is a finite separable Galois cover of curves over an algebraically closed field $k$ of characteristic $p > 0$ and $y\in Y$ is a ramification point with image $x\in X$ such that the inertia group $I(y\mid x)$ is $\Z/p\Z$. Then \'{e}tale-locally, $\varphi$ factors through an Artin--Schreier root stack $\wp_{m}^{-1}((L,s,f)/X)$. 
\end{thm}

This says that cyclic $p$-covers of curves $Y\rightarrow X$ yield quotient stacks that have an Artin--Schreier root stack structure. By Artin--Schreier theory, over an algebraically closed field of characteristic $p$ there are infinitely many non-isomorphic curves $Y$ covering $\P^{1}$ with Galois group $\Z/p\Z$, so the theorem implies we have infinitely many non-isomorphic stacky curves over $\P^{1}$ with a single stacky point of order $p$. This is a phenomenon that only arises in positive characteristic. 

Next, if a stacky curve has an order $p$ stacky point, then locally about this point the stacky curve is isomorphic to an Artin--Schreier root stack: 

\begin{thm}[Theorem~\ref{thm:locallyASrootstack}(1)]
\label{thm:2}
Let $\X$ be a stacky curve over a perfect field $k$ of characteristic $p > 0$. If $\X$ contains a stacky point $x$ of order $p$, there is an open substack $\Zz\subseteq\X$ containing $x$ such that $\Zz\cong\wp_{m}^{-1}((L,s,f)/Z)$ where $(m,p) = 1$, $Z$ is an open subscheme of the coarse space $X$ of $\X$ and a triple $(L,s,f)$ on $Z$ as above. 
\end{thm}

Moreover, if the coarse space of $\X$ is $\P^{1}$, then $\X$ is the fibre product of finitely many global Artin--Schreier root stacks: 

\begin{thm}[Theorem~\ref{thm:locallyASrootstack}(2)]
\label{thm:3}
Suppose all the nontrivial stabilizers of $\X$ are cyclic of order $p$. If $\X$ has coarse space $X = \P^{1}$, then $\X$ is isomorphic to a fibre product of Artin--Schreier root stacks of the form $\wp_{m}^{-1}((L,s,f)/\P^{1})$ for some $(m,p) = 1$ and a triple $(L,s,f)$ on $\P^{1}$. 
\end{thm}

If the coarse space of $\X$ is not $\P^{1}$ however, then Theorem~\ref{thm:3} fails in general (see Example~\ref{ex:nonP1cex}). In fact, anytime the genus of the coarse space is at least $1$, there will be obstructions to sections $f$ inducing a global Artin--Schreier root stack structure on $\X$. 

The paper is organized as follows. In Sections~\ref{sec:stackycurves} and~\ref{sec:divisorsbundles}, we recall the definition and basic properties of stacky curves. A key feature of tame stacky curves is that they are locally isomorphic to a root stack; we review this construction, originally due to Cadman \cite{cad} and, independently, to Abramovich--Graber--Vistoli \cite{agv}, in Section~\ref{sec:kummerrootstack}. In Section~\ref{sec:ASrootstacks}, we define Artin--Schreier root stacks, describe their basic properties and use them to prove Theorems~\ref{thm:1}, \ref{thm:2} and~\ref{thm:3}. Finally, in Section~\ref{sec:canrings}, we compute the canonical divisor and canonical ring of some stacky curves in characteristic $p > 0$ and in Section~\ref{sec:future} we describe a potential application of these computations to the theory of modular forms in positive characteristic. 

The explicit definition of Artin--Schreier root stacks arose from email correspondence between the author and David Rydh, based on the latter's work with Andrew Kresch in 2010 on wild ramification in stacks, so the author would like to thank David for these many enlightening conversations, along with his suggestions for improvement of the original draft. The author is grateful to Vaidehee Thatte for suggesting Lemma~\ref{lem:ASinteqn}, and to the following people for numerous helpful comments: John Berman, Tom Mark, Soumya Sankar, John Voight and David Zureick-Brown. Finally, the author would like to thank his adviser, Andrew Obus, for his help editing this paper, which formed part of the author's doctoral thesis.


\section{Connections to Other Works}

It is known (cf.~\cite[11.3.1]{ols}) that under some mild hypotheses, a Deligne--Mumford stack is \'{e}tale-locally a quotient stack by the stabilizer of a point. More generally, the sequence of papers \cite{aov}, \cite{alp10}, \cite{alp13}, \cite{alp14}, \cite{ahr1}, \cite{ahr2} establishes a structure theory for (tame) algebraic stacks which says that a large class of algebraic stacks are \'{e}tale-locally a quotient stack of the form $[\Spec A/G_{x}]$ where $G_{x}$ is a linearly reductive stabilizer of a point $x\in |\X|$. The present article can be regarded as a small first step in the direction of a structure theory for wild stacks, although it is of a different flavor than the above papers. Our approach is more akin to the one taken in \cite{gs}. 

Our structure theory in Section~\ref{sec:ASrootstacks} also parallels the approach to wild ramification in formal orbifolds and parabolic bundles taken in \cite{kp}, \cite{kma} and \cite{kum}. There, the authors define a {\it formal orbifold} by specifying a smooth projective curve over an algebraically closed field together with a {\it branch data} abstractly representing the ramification data present in our construction. This allows one to relate formal orbifolds and, more importantly, a suitable notion of {\it orbifold bundle} on a formal orbifold, to the more classical notions of {\it parabolic covers} of curves and {\it parabolic bundles} (cf.~\cite{ms}). Formal orbifolds admit a Riemann--Hurwitz formula (\cite[Thm. 2.20]{kp}) analogous to Theorem~\ref{prop:wildstackyRH} and can be studied combinatorially as we did in Section~\ref{sec:ASrootstacks}. Moreover, they shed light on the \'{e}tale fundamental group of curves in arbitrary characteristic (cf.~\cite[Thm. 2.40]{kp} or \cite[Thm. 1.1]{kum}). The perspective we take is more of an ``organic'' algebro-geometric view of wildly ramified stacky curves which comes naturally out of the classification problem for Deligne--Mumford stacks in dimension $1$ discussed in Section~\ref{sec:intro}.


\section{Stacky Curves}
\label{sec:stackycurves}

Fix a base scheme $S$ and let $\X$ be a category fibred in groupoids over $\cat{Sch}_{S}$, equipped with a Grothendieck topology (usually the \'{e}tale topology). Every scheme $X\in\cat{Sch}_{S}$ is canonically identified with its functor of points $T\mapsto X(T) = \Hom_{S}(T,X)$ by the Yoneda embedding, which in turn can be identified with a category fibred in groupoids $X\rightarrow\cat{Sch}_{S}$ (cf.~\cite[3.2.2]{ols}: ``the $2$-Yoneda lemma''). We say $\X$ is: 
\begin{itemize}
\item a {\bf stack} if for every object $U\in\cat{Sch}_{S}$ and every covering $\{U_{i}\rightarrow U\}$, the induced morphism $\X(U)\rightarrow\X(\{U_{i}\rightarrow U\})$ is an equivalence of categories;

\item an {\bf algebraic stack} if in addition, the diagonal $\Delta : \X\rightarrow\X\times_{S}\X$ is representable and there is a smooth surjection $U\rightarrow\X$ where $U$ is a scheme;

\item a {\bf Deligne--Mumford stack} if it is algebraic and there exists an \'{e}tale surjection $U\rightarrow\X$ where $U$ is a scheme;

\item in the case when $S = \Spec k$, a {\bf tame stack} if the orders of its (necessarily finite) stabilizer groups are coprime to $\char k$; otherwise, $\X$ is said to be a {\bf wild stack};

\item a {\bf stacky curve} if it is a smooth, separated, connected, one-dimensional Deligne--Mumford stack which is generically a scheme, i.e.~there exists an open subscheme $U$ of the coarse moduli space $X$ of $\X$ such that the induced map $\X\times_{X}U\rightarrow U$ is an isomorphism. 
\end{itemize}

In the literature there are equivalent definitions of algebraic stacks in terms of groupoids (\cite[Tags 0CWR and 026O]{sp}) and groupoid fibrations (\cite[Def.~1.148]{beh}). Further, the definitions of algebraic and Deligne--Mumford stacks only require $U\rightarrow\X$ (or equivalently the diagonal $\Delta_{\X}$) to be representable by an algebraic space (cf.~\cite{ols}), but the distinction will not play a role in the present article since we deal only with Deligne--Mumford stacks. 

The set of points of a stack $\X$, denoted $|\X|$, is defined to be the set of equivalence classes of morphisms $x : \Spec k\rightarrow\X$, where $k$ is a field, and where two points $x : \Spec k\rightarrow\X$ and $x' : \Spec k'\rightarrow\X$ are equivalent if there exists a field $L\supseteq k,k'$ such that the diagram 
\begin{center}
\begin{tikzpicture}[scale=1.5]
  \node at (0,0) (a) {$\Spec L$};
  \node at (1,1) (b1) {$\Spec k$};
  \node at (1,-1) (b2) {$\Spec k'$};
  \node at (2,0) (c) {$\X$};
  \draw[->] (a) -- (b1);
  \draw[->] (a) -- (b2);
  \draw[->] (b1) -- (c) node[above,pos=.6] {$x$};
  \draw[->] (b2) -- (c) node[above,pos=.5] {$x'$};
\end{tikzpicture}
\end{center}
commutes. Then the {\bf stabilizer group} of a point $x\in |\X|$ is taken to be the pullback $G_{x}$ in the following diagram: 
\begin{center}
\begin{tikzpicture}[scale=2]
  \node at (0,1) (a) {$G_{x}$};
  \node at (1,1) (b) {$\Spec k$};
  \node at (0,0) (c) {$\X$};
  \node at (1,0) (d) {$\X\times_{S}\X$};
  \draw[->] (a) -- (b);
  \draw[->] (a) -- (c);
  \draw[->] (b) -- (d) node[right,pos=.5] {$(x,x)$};
  \draw[->] (c) -- (d) node[above,pos=.5] {$\Delta_{\X}$};
\end{tikzpicture}
\end{center}
As in scheme theory, a geometric point is a point $\bar{x} : \Spec k\rightarrow\X$ where $k$ is algebraically closed. 

\begin{ex}
Let $X$ be a scheme over a field $k$. Then $X$ is a Deligne--Mumford stack with trivial stabilizers everywhere. 
\end{ex}

\begin{ex}
\label{ex:quotientstack}
Suppose $G\leq\Aut(X)$ is a group (a smooth group scheme over $k$) acting on a smooth, quasi-projective scheme $X$. Define the {\it quotient stack} $[X/G]$ to be the category over $\cat{Sch}_{k}$ whose objects are triples $(T,P,\pi)$, where $T\in\cat{Sch}_{k}$, $P$ is a $G\times_{k}T$-torsor for the \'{e}tale site $T_{\et}$ and $\pi : P\rightarrow X\times_{k}T$ is a $G\times_{k}T$-equivariant morphism. Morphisms $(T',P',\pi')\rightarrow (T,P,\pi)$ in $[X/G]$ are given by compatible morphisms of $k$-schemes $\varphi : T'\rightarrow T$ and $G\times_{k}T'$-torsors $\psi : P'\rightarrow \varphi^{*}P$ such that $\varphi^{*}\pi\circ\psi = \pi'$. 
\end{ex}

\begin{prop}
If $G$ is finite, then $[X/G]$ is a Deligne--Mumford stack with \'{e}tale morphism $X\rightarrow [X/G]$ and coarse moduli space $[X/G]\rightarrow X/G$. More generally, $[X/G]$ is an algebraic stack which is Deligne--Mumford if and only if for every geometric point $\bar{x} : \Spec k\rightarrow [X/G]$, the stabilizer $G_{\bar{x}}$ is an \'{e}tale group scheme over $k$. 
\end{prop}

\begin{proof}
See \cite[8.1.12 and 8.4.2]{ols}. 
\end{proof}

\begin{lem}
\label{lem:locallyquotientstack}
Let $\X$ be a stacky curve over a field $k$ with coarse moduli space $\pi : \X\rightarrow X$ and fix a geometric point $\bar{x} : \Spec k\rightarrow\X$ with stabilizer $G_{\bar{x}}$. Then there is an \'{e}tale neighborhood $U\rightarrow X$ of $x := \pi\circ\bar{x}$ and a finite morphism of schemes $V\rightarrow U$ such that $G_{\bar{x}}$ acts on $V$ and $\X\times_{X}U \cong [V/G_{\bar{x}}]$ as stacks. 
\end{lem}

\begin{proof}
This is \cite[11.3.1]{ols}.
\end{proof}

This says that that every stacky curve is, \'{e}tale locally, a quotient stack $[U/G]$ for $U$ a scheme and $G = G_{\bar{x}}$ the stabilizer of a geometric point (and thus a constant group scheme). By ramification theory (cf.~\cite[Ch.~IV]{ser}), if $\X$ is a tame stacky curve then every stabilizer group of $\X$ is cyclic. Consequently, with some hypotheses (cf. \cite{gs}) a tame stacky curve can be described completely by specifying its coarse moduli space and a finite list of numbers corresponding to the finitely many points with nontrivial stabilizers of those orders. In contrast, if $\X$ is wild, it may have higher ramification data and even nonabelian stabilizers. The main goal of this article is to describe how wild stacky curves can still be classified. 

Let $\X$ be a stacky curve, $\bar{x} : \Spec k\rightarrow\X$ a geometric point with image $x$ and stabilizer $G_{x}$ and let $\mathcal{G}_{x}\rightarrow X$ be the residue gerbe at $x$, i.e.~the unique monomorphism from a gerbe $\mathcal{G}$ with $|\mathcal{G}|$ a point through which $\bar{x}$ factors. We say $x$ is a {\it stacky point} of $\X$ if $G_{x}\not = 1$. As a substack of $\X$, this $\mathcal{G}_{x}$ may be regarded as a ``fractional point'', in the sense that $\deg\mathcal{G}_{x} = \frac{1}{|G_{x}|}$. 

The following technical lemma will be useful for constructing isomorphisms of stacks in later sections. 

\begin{lem}
\label{lem:CFGequiv}
If $F : \X\rightarrow\Y$ is a functor between categories fibred in groupoids over $\cat{Sch}_{S}$, then $F$ is an equivalence of categories fibred in groupoids if and only if for each $S$-scheme $T$, the functor $F_{T} : \X(T)\rightarrow\Y(T)$ is an equivalence of categories. 
\end{lem}

\begin{proof}
This is a special case of~\cite[Tag 003Z]{sp}. 
\end{proof}

Recall that for a locally noetherian scheme $X$ over $S$, the {\it relative normalization} of $X$ is an $S$-scheme $X^{\nu}$ together with an $S$-morphism $X^{\nu}\rightarrow X$ uniquely determined by the following properties: 
\begin{enumerate}[\quad (1)]
  \item $X^{\nu}\rightarrow X$ is integral, surjective and induces a bijection on irreducible components. 
  \item $X^{\nu}\rightarrow X$ is terminal among morphisms of $S$-schemes $Z\rightarrow X$ where $Z$ is normal. 
\end{enumerate}
Equivalently, $X^{\nu}$ is the normalization of $X$ in its total ring of fractions (cf.~\cite[Tag 035E]{sp}). 

\begin{lem}[{\cite[Tag 07TD]{sp}}]
If $Y\rightarrow X$ is a smooth morphism of locally noetherian $S$-schemes and $X^{\nu}$ is the relative normalization of $X$, then $X^{\nu}\times_{X}Y$ is the relative normalization of $Y$. 
\end{lem}

As a consequence we can extend the definition of normal/normalization to an algebraic stack. 

\begin{defn}
Let $\X$ be a locally noetherian algebraic stack over $S$. Then $\X$ is {\bf normal} if there is a smooth presentation $U\rightarrow\X$ where $U$ is a normal scheme. The {\bf relative normalization} of $\X$ is an algebraic stack $\X^{\nu}$ and a representable morphism of stacks $\X^{\nu}\rightarrow\X$ such that for any smooth morphism $U\rightarrow\X$ where $U$ is a scheme, $U\times_{\X}\X^{\nu}$ is the relative normalization of $U\rightarrow S$. 
\end{defn}

\begin{lem}[{\cite[Lem.~A.5]{ab}}]
For a locally noetherian algebraic stack $\X$, the relative normalization $\X^{\nu}$ is uniquely determined by the following two properties: 
\begin{enumerate}[\quad (1)]
  \item $\X^{\nu}\rightarrow\X$ is an integral surjection which induces a bijection on irreducible components. 
  \item $\X^{\nu}\rightarrow\X$ is terminal among morphisms of algebraic stacks $\Zz\rightarrow\X$, where $\Zz$ is normal, which are dominant on irreducible components. 
\end{enumerate}
\end{lem}

For more properties of relative normalizations of stacks, see \cite[Appendix A]{ab}. 

\begin{defn}
\label{defn:normalpullback}
Let $\X,\Y$ and $\Zz$ be algebraic stacks and suppose there are morphisms $\Y\rightarrow\X$ and $\Zz\rightarrow\X$. Define the {\bf normalized pullback} $\Y\times_{\X}^{\nu}\Zz$ to be the relative normalization of the fibre product $\Y\times_{\X}\Zz$. 
\end{defn}

We will write the normalized pullback as a diagram 
\begin{center}
\begin{tikzpicture}[scale=2]
  \node at (0,1) (a) {$\Y\times_{\X}^{\nu}\Zz$};
  \node at (1,1) (b) {$\Zz$};
  \node at (0,0) (c) {$\Y$};
  \node at (1,0) (d) {$\X$};
  \draw[->] (a) -- (b);
  \draw[->] (a) -- (c);
  \draw[->] (b) -- (d);
  \draw[->] (c) -- (d);
  \node at (.3,.7) {$\nu$};
  \draw (.2,.6) -- (.4,.6) -- (.4,.8);
\end{tikzpicture}
\end{center}


\section{Divisors and Vector Bundles}
\label{sec:divisorsbundles}

Let $\X$ be a normal Deligne--Mumford stack with coarse space $X$. As in scheme theory, we make the following definitions: 
\begin{itemize}
\item An {\bf irreducible (Weil) divisor} on $\X$ is an irreducible, closed substack of $\X$ of codimension $1$. 

\item The {\bf (Weil) divisor group} of $\X$, denoted $\Div\X$, is the free abelian on the irreducible divisors of $\X$; its elements are called {\bf (Weil) divisors} on $\X$. 

\item A {\bf principal (Weil) divisor} on $\X$ is a divisor $\div(f)$ associated to a morphism $f : \X\rightarrow\P_{k}^{1}$ (equivalently, a rational section $f$ of $\orb_{\X}$) given by 
$$
\div(f) = \sum_{Z} v_{Z}(f)Z
$$
where $v_{Z}(f)$ is the valuation of $f$ in the local ring $\orb_{X,Z}$. 

\item Two divisors $D,D'\in\Div\X$ are {\bf linearly equivalent} if $D = D' + \div(f)$ for some morphism $f : \X\rightarrow\P_{k}^{1}$. 

\item The subgroup of principal divisors in $\Div\X$ is denoted $\PDiv\X$. The {\bf divisor class group} of $\X$ is the quotient group $\Cl(\X) = \Div\X/\PDiv\X$. 
\end{itemize}

\begin{lem}
\label{lem:prindivcoarse}
Let $\X$ be a stacky curve over $k$ with coarse space morphism $\pi : \X\rightarrow X$. Then for any nonconstant map $f : \X\rightarrow\P_{k}^{1}$, $\div(f) = \pi^{*}\div(f')$, where $f' : X\rightarrow\P_{k}^{1}$ is the unique map making the diagram 
\begin{center}
\begin{tikzpicture}[scale=2]
  \node at (-.7,1) (a) {$\X$};
  \node at (.7,1) (b) {$X$};
  \node at (0,0) (c) {$\P_{k}^{1}$};
  \draw[->] (a) -- (b) node[above,pos=.5] {$\pi$};
  \draw[->] (a) -- (c) node[left,pos=.5] {$f$};
  \draw[->,dashed] (b) -- (c) node[right,pos=.5] {$f'$};
\end{tikzpicture}
\end{center}
commute. 
\end{lem}

\begin{proof}
See \cite[5.4.4]{vzb}. 
\end{proof}

For a scheme $X$, recall that a {\it rational divisor} on $X$ is a formal sum $E = \sum_{Z\subset X} r_{Z}Z$ where $r_{Z}\in\Q$ for each irreducible divisor $Z\subset X$. Denote the abelian group of rational divisors on $X$ by $\Q\Div X$. 

\begin{prop}
\label{prop:Qdivcorresp}
Let $\X$ be a stacky curve with stacky points $P_{1},\ldots,P_{n}$ whose stabilizers have orders $m_{1},\ldots,m_{n}$, respectively, and let $X$ be the coarse space of $\X$. Then there is a one-to-one correspondence between divisors $D\in\Div\X$ and rational divisors $E = \sum_{P} r_{P}P\in\Q\Div X$ such that $m_{i}r_{P_{i}}\in\Z$ for each $1\leq i\leq n$. 
\end{prop}

\begin{proof}
This is \cite[Thm. 1.187]{beh}. 
\end{proof}

\begin{defn}
The {\bf degree} of a divisor $D = \sum_{P} n_{P}P$ on a stacky curve is the formal sum $\deg(D) = \sum_{P} n_{P}\deg P$, where $\deg P = \deg\mathcal{G}_{P} = \frac{1}{|G_{P}|}$ is the degree of the residue gerbe $\mathcal{G}_{P}$ at $P$. 
\end{defn}

\begin{rem}
The degree of a divisor on a higher dimensional normal Deligne--Mumford stack is also defined, though we will not need it here. In general, a divisor on $\X$ need not have integer degree or coefficients, as the example below shows. However, Lemma~\ref{lem:prindivcoarse} shows that for a morphism $f : \X\rightarrow\P_{k}^{1}$ (i.e.~a rational section of $\orb_{\X}$), the principal divisor $\div(f)$ always does. 
\end{rem}

\begin{ex}
Let $n\geq 1$ and consider the quotient stack $\X = [\A_{\C}^{1}/\mu_{n}]$, where $\mu_{n}$ is the finite group scheme of $n$th roots of unity acting on $\A_{\C}^{1}$ by multiplication. Then $\X$ has coarse space $X = \A_{\C}^{1}/\mu_{n} = \A_{\C}^{1}$; let $\pi : \X\rightarrow X$ be the coarse map. Consider the morphism of sheaves $\pi^{*}\Omega_{\X}\rightarrow\Omega_{\X}$, where $\Omega_{\X}$ is the sheaf of differentials (defined below). Here, $\Omega_{\A_{\C}^{1}}$ is freely generated by $dt$ where $\A_{\C}^{1} = \Spec\C[t]$. Let $z$ be a coordinate on the upstairs copy of $\A_{\C}^{1}$, so that the cover $\A_{\C}^{1}\rightarrow [\A_{\C}^{1}/\mu_{n}]\rightarrow \A_{\C}^{1}$ is given by $t\mapsto z^{n}$. Since $\A_{\C}^{1}\rightarrow [\A_{\C}^{1}/\mu_{n}]$ is \'{e}tale, we can identify $\Omega_{\X}$ with $\Omega_{\A_{\C}^{1}}$ generated by $dz$. Then the section $s$ of $\Omega_{\X}\otimes\pi^{*}\Omega_{X}^{\vee}$ given by $\pi^{*}dt\mapsto d(z^{n}) = nz^{n - 1}\, dt$ has divisor $\div(s) = (n - 1)P$, where $P$ is the stacky point over the origin in $\A_{\C}^{1}$. Thus $\deg(\div(s)) = (n - 1)\deg(P) = \frac{n - 1}{n}$, so divisors of sections of nontrivial line bundles can have non-integer coefficients. 
\end{ex}

\begin{defn}
A {\bf Cartier divisor} on $\X$ is a Weil divisor which is (\'{e}tale-)locally of the form $(f)$ for a rational section $f$ of $\orb_{\X}$. 
\end{defn}

\begin{lem}
If $\X$ is a smooth Deligne--Mumford stack, then every Weil divisor is a Cartier divisor. 
\end{lem}

\begin{proof}
Pass to a smooth presentation and use~\cite[Lem.~3.1]{gs}. 
\end{proof}

\begin{lem}
\label{lem:linebundleOD}
Every line bundle $L$ on a stacky curve $\X$ is isomorphic to $\orb_{\X}(D)$ for some divisor $D\in\Div(\X)$. Moreover, $\orb_{\X}(D) \cong \orb_{\X}(D')$ if and only if $D$ and $D'$ are linearly equivalent. 
\end{lem}

\begin{proof}
This is \cite[5.4.5]{vzb}. 
\end{proof}

\begin{rem}
Lemma~\ref{lem:linebundleOD} says that for a stacky curve $\X$, the homomorphism $\Cl(\X)\rightarrow\Pic(\X),D\mapsto\orb_{\X}(D)$ is an isomorphism. This isomorphism holds more generally for locally factorial schemes which are reduced (cf. \cite[Tag 0BE9]{sp}), and it should similarly hold for reduced, locally factorial stacks, though we were not able to find a reference. 
\end{rem}

\begin{lem}
\label{lem:floor}
Let $\X$ be a Deligne--Mumford stack and $D\in\Div\X$. Then there is an isomorphism of sheaves (on $X$)
$$
\pi_{*}\orb_{\X}(D) \cong \orb_{X}(\lfloor D\rfloor)
$$
where $\lfloor D\rfloor$ is the floor divisor obtained by rounding down all $\Q$-coefficients of $D$, considered as an integral divisor on $X$. 
\end{lem}

\begin{proof}
For an \'{e}tale morphism $U\rightarrow X$, define the map 
\begin{align*}
  \orb_{X}(\lfloor D\rfloor)(U) &\longrightarrow \pi_{*}\orb_{\X}(D)(U)\\
    f &\longmapsto f\circ p_{2}
\end{align*}
where $p_{2} : \X\times_{X}U\rightarrow U$ is the canonical projection. A section $f\in\orb_{X}(\lfloor D\rfloor)(U)$ is a rational function on $\X$ satisfying $\lfloor D\rfloor + \div(f) \geq 0$, or equivalently, $D + \div(f\circ p_{2}) \geq 0$. Indeed, by Lemma~\ref{lem:prindivcoarse}, $\pi^{*}\div(f) = \div(f\circ p_{2})$ and the floor function does not change effectivity. Moreover, it is clear that $f\mapsto f\circ p_{2}$ is injective. For a map $g : \X\times_{X}U\rightarrow\P^{1}$, there is a map $f : U\rightarrow\P^{1}$ making the diagram 
\begin{center}
\begin{tikzpicture}[scale=2]
  \node at (-.7,1) (a) {$\X\times_{X}U$};
  \node at (.7,1) (b) {$U$};
  \node at (0,0) (c) {$\P^{1}$};
  \draw[->] (a) -- (b) node[above,pos=.5] {$p_{2}$};
  \draw[->] (a) -- (c) node[left,pos=.5] {$g$};
  \draw[->] (b) -- (c) node[right,pos=.5] {$f$};
\end{tikzpicture}
\end{center}
commute by the universal property of $p_{2}$. Therefore $f\mapsto f\circ p_{2}$ is surjective, so we get the desired isomorphism $\orb_{X}(\lfloor D\rfloor) \xrightarrow{\sim} \pi_{*}\orb_{\X}(D)$. 
\end{proof}

For any morphism of Deligne--Mumford stacks $f : \X\rightarrow\Y$, define the {\it sheaf of relative differentials} $\Omega_{\X/\Y}$ to be the sheaf which takes an \'{e}tale map $U\rightarrow\X$ to $\Omega_{U/k}(U)$, the vector space of $k$-differentials over $U$. Set $\Omega_{\X} := \Omega_{\X/\Spec k}$. (Equivalently, $\Omega_{\X/\Y}$ can be defined as the sheafification of the presheaf $U\mapsto\Omega_{\orb_{\X}(U)/f^{-1}\orb_{\Y}(U)}$). 

\begin{lem}
For a morphism of Deligne--Mumford stacks $f : \X\rightarrow\Y$, let $\mathcal{I}\subseteq\orb_{\X\times_{\Y}\X}$ be the ideal sheaf corresponding to the diagonal morphism $\Delta : \X\rightarrow\X\times_{\Y}\X$. Then $\Omega_{\X/\Y} \cong \Delta^{*}\mathcal{I}/\mathcal{I}^{2}$. 
\end{lem}

\begin{proof}
Standard. 
\end{proof}

%

\begin{cor}
Let $\X$ be a stacky curve over $k$. Then $\Omega_{\X}$ is a line bundle. 
\end{cor}

\begin{proof}
For any \'{e}tale map $U\rightarrow\X$, $f^{*}\Omega_{\X}\cong\Omega_{U}$ is a line bundle. Moreover, $\Omega_{\X}$ is locally free and rank of locally free sheaves is preserved under $f^{*}$, so $\Omega_{\X}$ is itself a line bundle. 
\end{proof}

As a result, Lemma~\ref{lem:linebundleOD} shows that $\Omega_{\X} \cong \orb_{\X}(K_{\X})$ for some divisor $K_{\X}\in\Div\X$, called a {\it canonical divisor} of $\X$. Note that a canonical divisor is unique up to linear equivalence. 

\begin{prop}[Stacky Riemann--Hurwitz -- Tame Case]
\label{prop:stackyRH}
For a tame stacky curve $\X$ with coarse moduli space $\pi : \X\rightarrow X$, the formula 
$$
K_{\X} = \pi^{*}K_{X} + \sum_{x\in\X(k)} (|G_{x}| - 1)x
$$
defines a canonical divisor $K_{\X}$ on $\X$. 
\end{prop}

\begin{proof}
This is \cite[Prop.~5.5.6]{vzb}, but the technique will be useful in a later generalization so we paraphrase it here. First assume $\X$ has a single stacky point $P$. Since $\pi$ is an isomorphism away from the stacky points, we may assume by Lemma~\ref{lem:locallyquotientstack} that $\X$ is of the form $\X = [U/\mu_{r}]$ for a scheme $U$ and $\mu_{r}$ the group scheme of $r$th roots of unity, where $r$ is coprime to $\char k$. Set $X = U/\mu_{r}$ and consider the \'{e}tale cover $f : U\rightarrow [U/\mu_{r}]$. Then $f^{*}\Omega_{\X/X}\cong\Omega_{U/X}$ so the stalk of $\Omega_{\X/X}$ at the stacky point has length $r - 1$. In general, since $\X\rightarrow X$ is an isomorphism away from the stacky points, the above calculation at each stacky point implies that the given formula for $K_{\X}$ defines a canonical divisor globally. 
\end{proof}

For a tame stacky curve $\X$, define its {\it (topological) Euler characteristic} $\chi(\X) = -\deg(K_{\X})$ and define its {\it genus} $g(\X)$ by the equation $\chi(\X) = 2 - 2g(\X)$. 

\begin{cor}
\label{cor:stackyRH}
For a tame stacky curve $\X$ with coarse space $X$, 
$$
g(\X) = g(X) + \frac{1}{2}\sum_{x\in\X(k)} \left (1 - \frac{1}{|G_{x}|}\right )\deg\pi(x). 
$$
where $\pi : \X\rightarrow X$ is the coarse map. 
\end{cor}


\section{Line Bundles, Sections and Root Stacks}
\label{sec:kummerrootstack}

Suppose $L$ is a line bundle on a scheme $X$. A natural question to ask is whether there exists another line bundle, say $E$, such that $E^{r} := E^{\otimes r} = L$ for a given integer $r\geq 1$. The following construction, originally found in \cite{cad} and \cite{agv}, produces a stacky version of $X$ called a root stack on which objects like $L^{1/r}$ live. An important application for our purposes is that every tame stacky curve over an algebraically closed field is a root stack. The authors in \cite{vzb} use this to give a complete description of the canonical ring of tame stacky curve, which will be the subject of Section~\ref{sec:canrings}. 

Recall that in the category of topological spaces, for any group $G$ there is a principal $G$-bundle functor 
$$
X \longmapsto \widetilde{\Bun}_{G}(X) = \{\text{principal $G$-bundles } P\rightarrow X\}/\text{iso.}
$$
which is represented by a classifying space $BG$, i.e.~there is a natural isomorphism 
$$
\widetilde{\Bun}_{G}(-) \cong [-,BG]
$$
where $[X,BG]$ denotes the set of homotopy classes of maps $X\rightarrow BG$. Topologically, $BG$ is the base of a universal bundle $EG\rightarrow BG$ with contractible total space $EG$, so informally, we can think of the classifying space as $BG = \bullet/G$. In fact, this works algebraically: set $EG = \Spec k$ (for $k$-schemes, and $EG = S$ for schemes over an arbitrary base $S$) and $BG = EG/G$ where $G$ acts trivially on $EG$. Additionally, it is often desirable to have a classification of \emph{all} principal $G$-bundles over $X$, rather than just isomorphism classes of bundles. 

With these things in mind, the best object for understanding principal bundles in algebraic geometry is the {\it classifying stack} $BG = [EG/G]$, where as above, $EG = \Spec k$ and $G$ is a group scheme over $k$. (From here on, let $BG$ denote the classifying stack; there should be no confusion with the topological space discussed above.) For a fixed $k$-scheme $X$, the assignment $T\mapsto \Bun_{G}(X)(T) = \{\text{principal $G$-bundles over } T\times_{k}X\}$ defines a stack (\cite[VIII, Thm.~1.1 and Prop.~1.10]{sga1}), and we have: 

\begin{prop}
For all $k$-schemes $X$, there is an isomorphism of stacks
$$
\Bun_{G}(X) \xrightarrow{\;\sim\;} \Hom_{\cat{Stacks}}(X,BG)
$$
where $\Hom_{\cat{Stacks}}(-,-)$ denotes the internal $\Hom$ stack in the category of $k$-stacks. 
\end{prop}

\begin{proof}
Follows from the definition of $BG = [\Spec k/G]$ in Example~\ref{ex:quotientstack}. 
\end{proof}

\begin{ex}
\label{ex:vectorbundle}
When $G = GL_{n}(k)$, principal $G$-bundles are in one-to-one correspondence with rank $n$ vector bundles (locally free sheaves) on $X$. Any line bundle $L\rightarrow X$ determines a $\G_{m}$-bundle $L_{0} = L\smallsetminus s_{0}(X)$, where $s_{0} : X\rightarrow L$ is the zero section. (Equivalently, $L_{0}$ is the frame bundle of $L$.) Conversely, a $\G_{m}$-bundle $P\rightarrow X$ determines a line bundle by the associated bundle construction: 
$$
L = P\times^{\G_{m}}\A^{1} := \{(y,\lambda)\in P\times\A^{1}\}/(y\cdot g,\lambda)\sim (y,g\lambda) \text{ for } g\in\G_{m}. 
$$
(We will use the convention that a group acts on principal bundles on the right.) One can check that $L\mapsto L_{0}$ and $P\mapsto P\times^{\G_{m}}\A^{1}$ are mutual inverses. In particular, principal $\G_{m}$-bundles are identified with line bundles on $X$, and this correspondence extends to an isomorphism of stacks $\Pic(X) \xrightarrow{\sim} \Hom_{\cat{Stacks}}(X,B\G_{m})$, where $\Pic(X)$ is the Picard stack on $X$ (cf.~\cite[Tag 0372]{sp}). 
\end{ex}

For a scheme $X$, let $\DIV^{[1]}(X)$ denote the category whose objects are pairs $(L,s)$, with $L\rightarrow X$ a line bundle and $s\in H^{0}(X,L)$ is a global section. A morphism $(L,s)\rightarrow (M,t)$ in $\DIV^{[1]}(X)$ is given by a bundle isomorphism 
\begin{center}
\begin{tikzpicture}[scale=2]
  \node at (-.7,1) (a) {$L$};
  \node at (.7,1) (b) {$M$};
  \node at (0,0) (c) {$X$};
  \draw[->] (a) -- (b) node[above,pos=.5] {$\varphi$};
  \draw[->] (a) -- (c);
  \draw[->] (b) -- (c);
\end{tikzpicture}
\end{center}
under which $\varphi(s) = t$. The notation $\DIV^{[1]}(X)$ is adapted from the notation $\DIV^{+}(X)$ used in some places in the literature, e.g.~\cite{ols}. We make the change to allow for a generalization in Section~\ref{sec:ASrootstacks}. 

By Example~\ref{ex:vectorbundle}, $B\G_{m}$ classifies line bundles, but to classify pairs $(L,s)$, we need to add a little ``fuzz'' to $B\G_{m}$. The next result shows how to do this with a quotient stack that is a ``thickened'' version of $B\G_{m}$. 

\begin{prop}
\label{prop:linebundlesectionclassification}
There is an isomorphism of categories fibred in groupoids $\DIV^{[1]} \cong [\A^{1}/\G_{m}]$. 
\end{prop}

\begin{proof}
Let $[\A^{1}/\G_{m}]\rightarrow B\G_{m}$ be the ``forgetful map'', sending an $X$-point 
$$
\left (\tikz[scale=1.2,baseline=12]{
  \node at (0,1) (a) {$P$};
  \node at (1,1) (b) {$\A^{1}$};
  \node at (0,0) (c) {$X$};
  \draw[->] (a) -- (b) node[above,pos=.5] {$f$};
  \draw[->] (a) -- (c) node[left,pos=.5] {$\pi$};
}\right )\in [\A^{1}/\G_{m}](X)
$$
to the $\G_{m}$-bundle $(P\xrightarrow{\pi}X)\in B\G_{m}(X)$. By the universal property of pullbacks, a map $X\rightarrow [\A^{1}/\G_{m}]$ is equivalent to the choice of a map $g : X\rightarrow B\G_{m}$ and a section $s$ of $X\times_{B\G_{m}}[\A^{1}/\G_{m}]\rightarrow X$. Further, $g$ is equivalent to a $\G_{m}$-bundle $E = g^{*}E\G_{m}\rightarrow X$, where $E\G_{m}\rightarrow B\G_{m}$ is the universal $\G_{m}$-bundle, and $L = X\times_{B\G_{m}}[\A^{1}/\G_{m}]$ is the line bundle on $X$ corresponding to $E$ (as in Example~\ref{ex:vectorbundle}), so $[\A^{1}/\G_{m}](X)$ is in bijection with $\DIV^{[1]}(X)$. All of the above choices are natural, so we have constructed an equivalence of categories on fibres $\DIV^{[1]}(X)\xrightarrow{\sim}[\A^{1}/\G_{m}](X)$ for each $X$. By definition, $\DIV^{[1]}$ is a category fibred in groupoids (over $\cat{Sch}$), so by Lemma~\ref{lem:CFGequiv}, this defines an isomorphism of categories fibred in groupoids $\DIV^{[1]}\xrightarrow{\sim} [\A^{1}/\G_{m}]$. 
\end{proof}

\begin{cor}
$\DIV^{[1]}$ is an algebraic stack. 
\end{cor}

We can similarly classify sequences of pairs $(L_{1},s_{1}),\ldots,(L_{n},s_{n})$ by a single quotient stack. 

\begin{lem}
\label{lem:nfoldproductA1Gm}
For all $n\geq 2$, $[\A^{n}/\G_{m}]\cong \underbrace{[\A^{1}/\G_{m}]\times_{B\G_{m}}\cdots\times_{B\G_{m}}[\A^{1}/\G_{m}]}_{n}$. 
\end{lem}

\begin{proof}
An $X$-point of the $n$-fold product $[\A^{1}/\G_{m}]\times_{B\G_{m}}\cdots\times_{B\G_{m}}[\A^{1}/\G_{m}]$ is the same thing as a $\G_{m}$-bundle $P\rightarrow X$ with a collection of equivariant maps $f_{1},\ldots,f_{n} : P\rightarrow\A^{1}$, but this data is equivalent to the same bundle $P\rightarrow X$ with a single map $(f_{1},\ldots,f_{n}) : P\rightarrow\A^{n}$, i.e. an $X$-point of the quotient stack $[\A^{n}/\G_{m}]$. 
\end{proof}

Let us now return to the question of when an $r$th root of a line bundle exists. This can alternatively be phrased in terms of cyclic $G$-covers (when the characteristic of the ground field does not divide $|G|$), and Kummer theory gives a natural answer to the question. Recall that a Kummer extension of fields is a Galois extension $L/K$ with group $G = \Z/r\Z$. We will assume $(r,p) = 1$ when $\char K = p > 0$. Explicitly, every such extension is of the form 
$$
L = K[x]/(x^{r} - s) \quad\text{for some } s\in K^{\times}
$$
when $K$ contains all $r$th roots of unity, and the general case has a similar form. To understand cyclic extensions in the language of stacks, we have the following construction due independently to \cite{cad} and \cite{agv}. 

\begin{defn}
For $r\geq 1$, the {\bf universal Kummer stack} is the cover of stacks 
$$
r : [\A^{1}/\G_{m}] \longrightarrow [\A^{1}/\G_{m}]
$$
induced by $x\mapsto x^{r}$ on both $\A^{1}$ and $\G_{m}$. 
\end{defn}

\begin{defn}
For a scheme $X$, a line bundle $L\rightarrow X$ with section $s$ and an integer $r\geq 1$, the $r$th {\bf root stack} of $X$ along $(L,s)$, written $\sqrt[r]{(L,s)/X}$, is defined to be the pullback 
\begin{center}
\begin{tikzpicture}[xscale=3,yscale=2]
  \node at (0,1) (a) {$\sqrt[r]{(L,s)/X}$};
  \node at (1,1) (b) {$[\A^{1}/\G_{m}]$};
  \node at (0,0) (c) {$X$};
  \node at (1,0) (d) {$[\A^{1}/\G_{m}]$};
  \draw[->] (a) -- (b);
  \draw[->] (a) -- (c);
  \draw[->] (b) -- (d) node[right,pos=.5] {$r$};
  \draw[->] (c) -- (d);
\end{tikzpicture}
\end{center}
where the bottom row is the morphism corresponding to $(L,s)$ via Proposition~\ref{prop:linebundlesectionclassification}. 
\end{defn}

\begin{rem}
\label{rem:rootstackpts}
Explicitly, for a test scheme $T$, the category $\sqrt[r]{(L,s)/X}(T)$ consists of tuples $(T\xrightarrow{\varphi}X,M,t,\psi)$ where $M\rightarrow T$ is a line bundle with section $t$ and $\psi : M^{r}\xrightarrow{\sim}\varphi^{*}L$ is an isomorphism of line bundles such that $\psi(t^{r}) = \varphi^{*}s$. 
\end{rem}

To take iterated roots, we need to extend our definition of $\DIV^{[1]}$ to Deligne--Mumford stacks. This is implicitly used in the literature but let us carefully spell things out here. For a Deligne--Mumford stack $\X$, let $\DIV^{[1]}(\X)$ be the collection of pairs $(\L,s)$ where $\L$ is a line bundle on $\X$ and $s$ is a \emph{section of $\L$}, a term which requires some explanation. By definition (cf.~\cite[1.171]{beh}), a vector bundle on a stack $\X$ is the data of a representable morphism of stacks $\E\rightarrow\X$ such that for all schemes $T\xrightarrow{f}\X$, the fibre product $\E_{T} := f^{*}\E\rightarrow T$ in the diagram 
\begin{center}
\begin{tikzpicture}[scale=2]
  \node at (0,1) (a) {$\E_{T}$};
  \node at (1,1) (b) {$\E$};
  \node at (0,0) (c) {$T$};
  \node at (1,0) (d) {$\X$};
  \draw[->] (a) -- (b);
  \draw[->] (a) -- (c);
  \draw[->] (b) -- (d);
  \draw[->] (c) -- (d) node[above,pos=.5] {$f$};
\end{tikzpicture}
\end{center}
is a vector bundle on $T$. In addition, we require the following naturality condition: for any morphism of schemes $\varphi : T'\rightarrow T$, there is a bundle map $\E_{T'} = (f\circ\varphi)^{*}\E\rightarrow \varphi^{*}f^{*}\E = \varphi^{*}\E_{T}$. (Alternatively, one can define vector bundles using the groupoid definition of stacks.) The sheaf of sections of a vector bundle $\E\rightarrow\X$ is given by $\Gamma(\X,E) : T\mapsto H^{0}(T,\E_{T})$. So a choice of `section' $s$ of a line bundle $\L\rightarrow\X$ is really a choice of section $s_{T}\in H^{0}(T,\E_{T})$ for each scheme $T\rightarrow\X$, compatible with morphisms $T'\rightarrow T$. As with schemes, there is a natural notion of morphisms $(\L,s)\rightarrow (\L',s')$ in $\DIV^{[1]}(\X)$. A direct consequence of Proposition~\ref{prop:linebundlesectionclassification} is the following. 

\begin{cor}
\label{cor:linebundlesectionclassification}
Let $\X$ be a Deligne--Mumford stack. There is an equivalence of categories 
$$
\DIV^{[1]}(\X) \xrightarrow{\;\sim\;} \Hom_{\cat{Stacks}}(\X,[\A^{1}/\G_{m}]). 
$$
\end{cor}

This extends the definition of a root stack to a stacky base: for a line bundle $\L\rightarrow\X$ over a stack $\X$ with section $s$, let $\sqrt[r]{(\L,s)/\X}$ be the pullback 
\begin{center}
\begin{tikzpicture}[xscale=3,yscale=2]
  \node at (0,1) (a) {$\sqrt[r]{(\L,s)/\X}$};
  \node at (1,1) (b) {$[\A^{1}/\G_{m}]$};
  \node at (0,0) (c) {$\X$};
  \node at (1,0) (d) {$[\A^{1}/\G_{m}]$};
  \draw[->] (a) -- (b);
  \draw[->] (a) -- (c);
  \draw[->] (b) -- (d) node[right,pos=.5] {$r$};
  \draw[->] (c) -- (d);
\end{tikzpicture}
\end{center}
where the bottom row comes from Corollary~\ref{cor:linebundlesectionclassification}. The following basic results and examples may be found in \cite{cad}. 

\begin{lem}
\label{lem:rootstacknatural}
For any morphism of stacks $h : \Y\rightarrow\X$ and line bundle $\L\rightarrow\X$ with section $s$, there is an isomorphism of root stacks 
$$
\sqrt[r]{(h^{*}\L,h^{*}s)/\Y} \xrightarrow{\;\;\sim\;\;} \sqrt[r]{(\L,s)/\X}\times_{\X}\Y. 
$$
\end{lem}

\begin{proof}
Follows easily from either the definition of root stack as a pullback, or from the description of its $T$-points above (Remark~\ref{rem:rootstackpts}). 
\end{proof}

\begin{ex}
\label{ex:rootstackA1}
Let $X$ be the affine line $\A^{1} = \Spec k[x]$. Then $L = \orb = \orb_{X}$ is a line bundle and the coordinate $x$ gives a section of $L$. Choose $r\geq 1$ that is coprime to $\char k$. We claim that $\sqrt[r]{(\orb,x)/\A^{1}} \cong [\A^{1}/\mu_{r}]$. By the comments above, for a test scheme $T$ the category $\sqrt[r]{(\orb,x)/\A^{1}}(T)$ consists of tuples $(T\xrightarrow{\varphi}\A^{1},M,t,\psi)$ where $M\rightarrow T$ is a line bundle with section $t$ and $\psi : M^{r}\xrightarrow{\sim}\varphi^{*}\orb_{X}$ sending $t^{r}\mapsto \varphi^{*}x$. Note that $\varphi^{*}\orb_{X} = \orb_{T}$ so that $\varphi^{*}x$ corresponds to a section $f\in H^{0}(T,\orb_{T})$ and thus determines a map $f : T\rightarrow\A^{1}$. On the other hand, $[\A^{1}/\mu_{r}](T)$ consists of principal $\mu_{r}$-bundles $P\rightarrow T$ together with $\mu_{r}$-equivariant morphisms $P\rightarrow\A^{1}$. Define a map $\sqrt[r]{(\orb,x)/\A^{1}}(T) \rightarrow [\A^{1}/\mu_{r}](T)$ by sending 
$$
(T\rightarrow X,M,t,\psi) \longmapsto \left (\tikz[scale=1.2,baseline=12]{
  \node at (0,1) (a) {$M_{0}$};
  \node at (1,1) (b) {$\A^{1}$};
  \node at (0,0) (c) {$T$};
  \draw[->] (a) -- (b) node[above,pos=.5] {$h$};
  \draw[->] (a) -- (c) node[left,pos=.5] {$\pi$};
}\right )
$$
where $M_{0}\rightarrow T$ and $h$ are obtained as follows. The Kummer sequence 
$$
\sesx{\mu_{r}}{\G_{m}}{\G_{m}}{}{r}
$$
induces an exact sequence 
$$
H^{1}(T,\mu_{r})\rightarrow H^{1}(T,\G_{m})\xrightarrow{r} H^{1}(T,\G_{m}). 
$$
By exactness, the choice of trivialization $\psi : M^{r}\xrightarrow{\sim}\orb_{T}$ determines a lift of $M$ to a $\mu_{r}$-bundle $M_{0}\rightarrow T$. Then $h$ is the map which takes a root of unity $\zeta$ in the fibre over $q\in T$ to the value $\zeta\sqrt[r]{f(q)}\in\A^{1}$, where $\sqrt[r]{f(q)}$ is a fixed $r$th root of $f(q)$ which is determined after choosing an explicit trivialization of $\varphi^{*}\orb_{X} = \orb_{T}$. (So in each fibre, $1\mapsto\sqrt[m]{f(q)}$.) This extends to an isomorphism of stacks 
$$
\sqrt[r]{(\orb,x)/\A^{1}} \xrightarrow{\;\;\sim\;\;} [\A^{1}/\mu_{r}]. 
$$
\end{ex}

\begin{ex}
\label{ex:elemrootstack}
More generally, when $X = \Spec A$ and $L = \orb_{X}$ with any section $s$, we have 
$$
\sqrt[r]{(\orb_{X},s)/X} \cong [\Spec B/\mu_{r}] \quad\text{where } B = A[x]/(x^{r} - s). 
$$
To see this, note that $(\orb_{X},s)$ induces a morphism $\Spec A\rightarrow [\A^{1}/\G_{m}]$ by Proposition~\ref{prop:linebundlesectionclassification}. Lifting along the $\G_{m}$-cover $\A^{1}\rightarrow [\A^{1}/\G_{m}]$, we get a map $F : X\rightarrow\A^{1}$ making the diagram 
\begin{center}
\begin{tikzpicture}[scale=2]
  \node at (0,0) (1) {$X$};
  \node at (1,1) (2) {$\A^{1}$};
  \node at (1,0) (3) {$[\A^{1}/\G_{m}]$};
  \draw[->,dashed] (1) -- (2) node[above,pos=.5] {$F$};
  \draw[->] (1) -- (3);
  \draw[->] (2) -- (3);
\end{tikzpicture}
\end{center}
commute and such that $\orb_{X} = F^{*}\orb_{\A^{1}}$ and $s = F^{*}x$. By Lemma~\ref{lem:rootstacknatural}, 
\begin{align*}
  \sqrt[r]{(\orb_{X},s)/X} &\cong \sqrt[r]{(\orb_{\A^{1}},x)/\A^{1}}\times_{\A^{1}}X\\
    &\cong [\A^{1}/\mu_{r}]\times_{\A^{1}}X \quad\text{by Example~\ref{ex:rootstackA1}}\\
    &\cong [\Spec B/\mu_{r}]
\end{align*}
for $B = k[x]/(x^{r} - s)\otimes_{k}A = A[x]/(x^{r} - s)$. In general, any root stack $\sqrt[r]{(L,s)/X}$ may be covered by such ``affine'' root stacks $[\Spec B/\mu_{r}]$: 
\end{ex}

\begin{prop}
\label{prop:coverbyrootstacks}
Let $\X = \sqrt[r]{(L,s)/X}$ be an $r$th root stack of a scheme $X$ along a pair $(L,s)$, with coarse map $\pi : \X\rightarrow X$. Then for any point $\bar{x} : \Spec k\rightarrow\X$, there is an affine \'{e}tale neighborhood $U = \Spec A\rightarrow X$ of $x = \pi\circ\bar{x}$ such that $U\times_{X}\X \cong [\Spec B/\mu_{r}]$, where $B = A[x]/(x^{r} - s)$. 
\end{prop}

\begin{proof}
This is an easy consequence of Lemma~\ref{lem:rootstacknatural} and Example~\ref{ex:elemrootstack}. 
\end{proof}

\begin{thm}
\label{thm:rootstackDM}
If $\X$ is a Deligne--Mumford stack with line bundle $\L$ and section $s$ and $r$ is invertible on $\X$, then $\sqrt[r]{(\L,s)/\X}$ is a Deligne--Mumford stack. 
\end{thm}

\begin{proof}
See \cite[2.3.3]{cad}. The technique of this proof will be used to prove an analogous result for Artin--Schreier root stacks (Theorem~\ref{thm:ASrootstackDM}). 
\end{proof}

\begin{ex}
If $s$ is a nonvanishing section of a line bundle $L\rightarrow X$ over a scheme, then $\sqrt[r]{(L,s)/X} \cong X$ as stacks. To see this, note that the following statements are equivalent: 
\begin{enumerate}[\quad (a)]
  \item $s$ is nonvanishing.
  \item $L$ is trivial.
  \item $L_{0}\cong X\times\G_{m}$ as principal bundles over $X$. 
  \item The induced map $X\rightarrow [\A^{1}/\G_{m}]$ factors through $X\rightarrow [\G_{m}/\G_{m}] = \Spec k$. 
\end{enumerate}
Further, (d) implies that $\sqrt[r]{(L,s)/X}\rightarrow X$ is an isomorphism, so (a) does as well. So for any pair $(L,s)$, the stacky structure of $\sqrt[r]{(L,s)/X}$ occurs precisely at the vanishing locus of $s$. 
\end{ex}

\begin{ex}
Let $L$ be a line bundle on a scheme $X$ and consider the pullback: 
\begin{center}
\begin{tikzpicture}[xscale=3,yscale=2]
  \node at (0,1) (a) {$X\times_{B\G_{m}}B\G_{m}$};
  \node at (1,1) (b) {$B\G_{m}$};
  \node at (0,0) (c) {$X$};
  \node at (1,0) (d) {$B\G_{m}$};
  \draw[->] (a) -- (b);
  \draw[->] (a) -- (c);
  \draw[->] (b) -- (d) node[right,pos=.5] {$r$};
  \draw[->] (c) -- (d) node[above,pos=.5] {$L$};
\end{tikzpicture}
\end{center}
Here, the bottom row is induced from the line bundle $L$, using that $B\G_{m}$ classifies line bundles. For the zero section $0$ of $L$, we can view the root stack $\sqrt[r]{(L,0)/X}$ as an infinitesimal thickening of this fibre product $X\times_{B\G_{m}}B\G_{m}$. To see this explicitly, note that by Proposition~\ref{prop:coverbyrootstacks}, $\sqrt[r]{(L,0)/X}$ may be covered by root stacks of the form $[\Spec B/\mu_{r}]$, where $B = A[x]/(x^{r} - 0) = A[x]/(x^{r}) = A\otimes_{k}k[x]/(x^{r})$ for some $\Spec A\subseteq X$. Thus $[\Spec B/\mu_{r}] \cong \Spec A\times B\mu_{r}$, which is indeed an infinitesimal thickening of $\Spec A\times_{B\G_{m}}B\G_{m}$. Alternatively: $X\times_{B\G_{m}}B\G_{m}$ is the closed substack of $\sqrt[r]{(L,0)/X}$ whose $T$-points for a scheme $T$ are given by 
$$
(X\times_{B\G_{m}}B\G_{m})(T) = \{(T\rightarrow X,M,t,\psi)\mid \psi(t^{r}) = 0\}. 
$$
In some places in the literature, the notation $\sqrt[r]{L/X}$ is used for the stack $X\times_{B\G_{m}}B\G_{m}$, in which case $\sqrt[r]{L/X}\hookrightarrow \sqrt[r]{(L,0)/X}$ is an intuitive notation for this infinitesimal thickening. 
\end{ex}

Note that if $p$ divides $r$, the root stack construction of \cite{cad} and \cite{agv} is still well-defined, but Theorem~\ref{thm:rootstackDM} fails. Thus to be able to study $p$th order (and more general) stacky structure in characteristic $p$, we must work with a different notion of root stack.


\section{Artin--Schreier Root Stacks}
\label{sec:ASrootstacks}

When $\char k = p > 0$ and we want to compute a $p$th root of a line bundle, the Frobenius immediately presents problems. Specifically, the cover $[\A^{1}/\G_{m}]\rightarrow [\A^{1}/\G_{m}]$ induced by $x\mapsto x^{p}$ is \emph{not} \'{e}tale. To remedy this, we can once again rephrase the question in terms of cyclic $G$-covers. This time, we will make use of Artin--Schreier theory when $G = \Z/p\Z$ and Artin--Schreier--Witt theory in the general case. Fix $G = \Z/p\Z$ and recall that an Artin--Schreier extension of a perfect field $K$ of characteristic $p$ is a Galois cover $L/K$ with group $\Z/p\Z$. These are explicitly of the form 
$$
L = K[x]/(x^{p} - x - a) \quad\text{for some } a\in K, a\not = b^{p} - b \text{ for any } b\in K. 
$$
When $K$ is a local field with valuation $v$, the integer $m = -v(a)$ is coprime to $p$ and is called the {\it ramification jump} of $L/K$. Different ramification jumps yield \emph{non-isomorphic} $\Z/p\Z$-extensions, a phenomenon which certainly doesn't occur in characteristic $0$. Moreover, Artin--Schreier extensions of a local field are completely classified up to isomorphism by their ramification jump (cf.~\cite[Ch.~IV]{ser}), so this discrete invariant is quite important to understanding $\Z/p\Z$-extensions in characteristic $p$. 

In the one-dimensional case, Artin--Schreier theory is also used to give a precise classification of Galois covers of curves $Y\rightarrow X$ with group $\Z/p\Z$. Suppose $k$ is algebraically closed and $X$ and $Y$ are $k$-curves. If $k(X)$ (resp. $k(Y)$) is the function field of $X$ (resp. $Y$) then the completion of $k(X)$ is isomorphic to the field of Laurent series $k((x))$. The extension $k((y))/k((x))$ given by the completion of $k(Y)$ is Galois of degree $p$, so it is given by an equation $y^{p} - y = f(x)$, and the ramification jump is $m = -v(f)$. We can represent $f(x) = x^{-m}g(x)$ for some $g\in k[[x]]$, and after a change of formal variables, we can even arrange for $g = 1$. When $k$ is algebraically closed, this shows there are infinitely many non-isomorphic $\Z/p\Z$-covers of any given curve $X$ since the ramification jump is an isomorphism invariant of the field extension. 

The following lemma will be useful in later arguments. 

\begin{lem}
\label{lem:ASinteqn}
Let $K = k((x))$ be the local field of Laurent series with valuation ring $A = k[[x]]$ and let $L/K$ be the $\Z/p\Z$-extension given by the equation $y^{p} - y = x^{-m}g(x)$ with $g\in A^{\times}$ $m\equiv -1\pmod{p}$. Write $m + 1= pn$ for $n\in\N$ and let $z = x^{n}y$. If $B$ denotes the integral closure of $A$ in $L$, then $B = A[z]$. 
\end{lem}

\begin{proof}
It is easy to see that $z$ satisfies the integral equation 
$$
z^{p} - zx^{n(p - 1)} = xg(x). 
$$
Next, note that $v_{L}(x) = pv_{K}(x) = p$ and $-mp = v_{L}(x^{-m}) = v_{L}(x^{-m}g) = v_{L}(y^{p} - y) = \min\{pv_{L}(y),v_{L}(y)\}$ which implies $v_{L}(y) = -m$. So $v_{L}(z) = nv_{L}(x) + v_{L}(y) = np - m = 1$. Hence $z$ is a uniformizer of $B$, so $B = A[z]$ as claimed. 
\end{proof}

\begin{rem}
\label{rem:ASinteqn}
When $m\not\equiv -1\pmod{p}$, it is not as easy to write down a normal integral equation for $B/A$. Write $m = pn - r$ for $1 < r < p$. Then $z = x^{n}y$ still satisfies the integral equation $z^{p} - zx^{n(p - 1)} = xg(x)$, but now $v_{L}(z) = r > 1$ so we don't get a uniformizer in $A[z]$ for free. In fact, $A[z]\not = B$ in these cases. To fix this, let $c,d\in\Z$ be the unique integers with $0 < c < p$ such that $cr - dp = 1$ and set $u = z^{c}x^{-d} = x^{nc - d}y^{c}$. Then $v_{L}(u) = cv_{L}(z) - dv_{L}(x) = cr - dp = 1$, so $A[u] = B$. However, it is difficult to write down the minimal polynomial of $u$ over $A$ and one should not expect it to produce a normal equation for $B/A$ in general. Instead, one can write down an integral basis of $B/A$ by resolving the singularities in any of the above integral equations step-by-step (such algorithms can be found in~\cite[Lem.~6.3]{op} or \cite[Lem.~5.5]{ls}). 
\end{rem}

Transitioning to the geometric question, suppose we have a line bundle $L\rightarrow X$ and a section $s\in H^{0}(X,L)$ of which we would like to find a $p$th root, i.e.~a pair $(E,t)$ with $(E^{\otimes p},t^{p}) \cong (L,s)$. By Proposition~\ref{prop:linebundlesectionclassification}, such pairs $(L,s)$ are classified by $X$-points of the quotient stack $[\A^{1}/\G_{m}]$. For our purposes, the algebraic stack $[\A^{1}/\G_{m}]$ is not sensitive enough to keep track of the extra information present in Artin--Schreier theory, namely the ramification jump. 

Instead, let $\P(1,m)$ be the \emph{weighted projective line} for an integer $m\geq 1$ which is coprime to $p$. Some authors view this as a scheme, in which case it is the projective line with homogeneous coordinates $[x,y]$ corresponding to the graded ring $k[x_{0},x_{1}]$, but with a generator $x_{0}$ in degree $1$ and a generator $x_{1}$ in degree $m$. However, it is more natural to view $\P(1,m)$ as a stack with a single nontrivial stabilizer group $\Z/m\Z$ at the point $\infty = [0,1]$. In particular, $\P(1,m)$ is a stacky curve and the natural morphism $\P(1,m)\rightarrow\P^{1}$ given by sending $[x,y]\mapsto [x,z]$, where $z = y^{m}$, is a coarse moduli map. The following shows two explicit constructions of $\P(1,m)$ as an algebraic stack. 

\begin{lem}
\label{lem:weightedprojrootstack}
For any $m\geq 1$ coprime to $p$, there are isomorphisms of stacks 
$$
[\A^{2}\smallsetminus\{0\}/\G_{m}] \cong \P(1,m) \cong \sqrt[m]{(\orb(1),s_{\infty})/\P^{1}}
$$
where in the first term, $\G_{m}$ acts on $\A^{2}\smallsetminus\{0\}$ by $\lambda\cdot (x,y) = (\lambda x,\lambda^{m}y)$, and in the third term, $s_{\infty}$ is the section whose divisor is the point $[0,1]$. Furthermore, $\P(1,m)$ is a Deligne--Mumford stack. 
\end{lem}

\begin{proof}
The first isomorphism is obvious from the definition of $\P(1,m)$ as a weighted projective space, while the second is clear from the structure of $\P(1,m)$ on the standard covering by affine opens, along with Proposition~\ref{prop:coverbyrootstacks}. Finally, Theorem~\ref{thm:rootstackDM} and the second isomorphism imply $\P(1,m)$ is Deligne--Mumford. 
\end{proof}

Fix $m\geq 1$ coprime to $p$ and let $\DIV^{[1,m]}(X)$ be the category consisting of triples $(L,s,f)$ with $L\in\Pic(X)$ and sections $s\in H^{0}(X,L)$ and $f\in H^{0}(X,L^{m})$ that don't vanish simultaneously. Morphisms $(L,s_{L},f_{L})\rightarrow (M,s_{M},f_{M})$ in $\DIV^{[1,m]}(X)$ are given by bundle isomorphisms 
\begin{center}
\begin{tikzpicture}[scale=2]
  \node at (-.7,1) (a) {$L$};
  \node at (.7,1) (b) {$M$};
  \node at (0,0) (c) {$X$};
  \draw[->] (a) -- (b) node[above,pos=.5] {$\varphi$};
  \draw[->] (a) -- (c);
  \draw[->] (b) -- (c);
\end{tikzpicture}
\end{center}
under which $\varphi(s_{L}) = s_{M}$ and $\varphi^{\otimes m}(f_{L}) = f_{M}$. Then $\DIV^{[1,m]}$ is a category fibred in groupoids over $\cat{Sch}_{k}$. In analogy with Proposition~\ref{prop:linebundlesectionclassification}, we have: 

\begin{prop}
\label{prop:ASlinebundlesectionclassification}
For each $m\geq 1$, there is an isomorphism of categories fibred in groupoids $\DIV^{[1,m]}\cong \P(1,m)$. 
\end{prop}

\begin{proof}
A map $X\rightarrow [\A^{2}\smallsetminus\{0\}/\G_{m}]$ is equivalent to a map $X\rightarrow [\A^{2}/\G_{m}]$ avoiding $(0,0)$, where $\G_{m}$ acts on $\A^{2}$ with weights $(1,m)$. Let $[\A^{2}/\G_{m}]\rightarrow B\G_{m}$ be the forgetful map. Then by the universal property of pullbacks, the map $X\rightarrow [\A^{2}/\G_{m}]$ is equivalent to the choice of a map $g : X\rightarrow B\G_{m}$ and a section $\sigma$ of the line bundle $L = X\times_{B\G_{m}}[\A^{2}/\G_{m}]\rightarrow X$. The proof of Lemma~\ref{lem:nfoldproductA1Gm} carries through when $\G_{m}$ acts on $\A^{2}$ with any weights, giving $[\A^{2}/\G_{m}]\cong [\A^{1}/\G_{m}]\times_{B\G_{m}}[\A^{1}/\G_{m}]$ with weights $(1,m)$. Then $\sigma$ really corresponds to a section $s$ of $L$, the line bundle associated to $g^{*}E\G_{m}$, and a section $f$ of $L^{m}$, the line bundle associated to $(g^{*}E\G_{m})^{m}$. All of these choices are natural, so we have constructed an equivalence of categories $\DIV^{[1,m]}(X)\xrightarrow{\sim}\P(1,m)(X)$ for each $X$. As in the proof of Proposition~\ref{prop:linebundlesectionclassification}, Lemma~\ref{lem:CFGequiv} guarantees that this extends to an isomorphism $\DIV^{[1,m]}\xrightarrow{\sim} \P(1,m)$ of categories fibred in groupoids. 
\end{proof}

\begin{cor}
$\DIV^{[1,m]}$ is a Deligne--Mumford stack. 
\end{cor}

\begin{rem}
It also follows from Lemma~\ref{lem:weightedprojrootstack} and Proposition~\ref{prop:ASlinebundlesectionclassification} that $\DIV^{[1,m]}$ is a root stack, namely $\sqrt[m]{(\orb(1),s_{\infty})/\P^{1}}$ but it will be useful for later arguments to exhibit this isomorphism directly, which we do now. For any scheme $X$, a functor 
$$
\Upsilon : \sqrt[m]{(\orb(1),s_{\infty})/\P^{1}}(X) \longrightarrow \DIV^{[1,m]}(X)
$$
can be built in the following way. For an object $\B = (X\xrightarrow{\varphi}\P^{1},L,t,L^{m}\xrightarrow{\sim}\varphi^{*}\orb(1))$ of the root stack, the morphism $\varphi$ induces two sections $\varphi^{*}s_{0},\varphi^{*}s_{\infty}\in H^{0}(X,\varphi^{*}\orb(1))$ and under the isomorphism $L^{m}\cong\varphi^{*}\orb(1)$, $\varphi^{*}s_{\infty}$ may be identified with $t^{m}$. Set $\Upsilon(\B) = (L,s,f)$ where $s = t = s_{\infty}^{1/m}$ and $f = s_{0}$. Naturality of $\Upsilon$ is clear from the definitions of morphisms in each category. One can check this gives the same isomorphism of stacks as Proposition~\ref{prop:ASlinebundlesectionclassification}. 
\end{rem}

As in Section~\ref{sec:kummerrootstack}, we extend the definition of $\DIV^{[1,m]}$ to a stacky base $\X$ by taking $\DIV^{[1,m]}(\X)$ to be the category of triples $(\L,s,f)$ where $\L$ is a line bundle on $\X$ and $s$ and $f$ are sections on $\L$ and $\L^{m}$, respectively, i.e.~a choice of section $s_{T}\in H^{0}(T,\E_{T})$ for each scheme $T\rightarrow\X$, compatible with morphisms $T'\rightarrow T$, and likewise for $f$. Then Proposition~\ref{prop:ASlinebundlesectionclassification} implies the following. 

\begin{cor}
\label{cor:ASlinebundlesectionclassification}
Let $\X$ be a Deligne--Mumford stack. For each $m\geq 1$, there is an equivalence of categories 
$$
\DIV^{[1,m]}(\X) \xrightarrow{\;\sim\;} \Hom_{\cat{Stacks}}(\X,\P(1,m)). 
$$
\end{cor}

The main feature of $\P(1,m)$ that makes it valuable to our program of study is the fact that the cyclic order $p$ isogeny $\wp : \G_{a}\rightarrow\G_{a},\alpha\mapsto \alpha^{p} - \alpha$ extends to a ramified cyclic $p$-cover $\Psi : \P(1,m)\rightarrow\P(1,m)$ given by $[u,v]\mapsto [u^{p},v^{p} - vu^{m(p - 1)}]$. Meanwhile, $\G_{a}$ acts on $\P(1,m)$ via $\alpha\cdot [u,v] = [u,v + \alpha u^{m}]$ and it is easy to check this action commutes with $\Psi$, so we get an induced morphism on the quotient stack $[\P(1,m)/\G_{a}]$. We now use $[\P(1,m)/\G_{a}]$ to construct a characteristic $p$ analogue of the root stack of \cite{cad} and \cite{agv}. 

\begin{defn}
Let $m\geq 1$ be coprime to $p$. The {\bf universal Artin--Schreier cover with ramification jump $m$} is the cover of stacks 
$$
\wp_{m} : [\P(1,m)/\G_{a}] \longrightarrow [\P(1,m)/\G_{a}]
$$
induced by $[u,v] \mapsto [u^{p},v^{p} - vu^{m(p - 1)}]$ on $\P(1,m)$ and $\alpha\mapsto\alpha^{p} - \alpha$ on $\G_{a}$. 
\end{defn}

The following definition is inspired by a short article \cite{ryd} by D.~Rydh and email correspondence between him and the author. The work in \cite{ryd} ultimately dates back to discussions between Rydh and A. Kresch in October 2010, and the present article might be viewed as a realization of some of the questions about wildly ramified stacks that first arose then. 

\begin{defn}
\label{defn:ASrootstack}
For a stack $\X$, a line bundle $\L\rightarrow\X$ and sections $s$ of $\L$ and $f$ of $\L^{m}$, the {\bf Artin--Schreier root stack} $\wp_{m}^{-1}((\L,s,f)/\X)$ is defined to be the normalized pullback of the diagram 
\begin{center}
\begin{tikzpicture}[xscale=4,yscale=2]
  \node at (0,1) (a) {$\wp_{m}^{-1}((\L,s,f)/\X)$};
  \node at (1,1) (b) {$[\P(1,m)/\G_{a}]$};
  \node at (0,0) (c) {$\X$};
  \node at (1,0) (d) {$[\P(1,m)/\G_{a}]$};
  \draw[->] (a) -- (b);
  \draw[->] (a) -- (c);
  \draw[->] (b) -- (d) node[right,pos=.5] {$\wp_{m}$};
  \draw[->] (c) -- (d);
  \node at (.15,.7) {$\nu$};
  \draw (.1,.6) -- (.2,.6) -- (.2,.8);
\end{tikzpicture}
\end{center}
where the bottom row is the composition of the morphism $\X\rightarrow\P(1,m)$ corresponding to $(\L,s,f)$ by Corollary~\ref{cor:ASlinebundlesectionclassification} and the quotient map $\P(1,m)\rightarrow [\P(1,m)/\G_{a}]$. 
\end{defn}

That is, $\wp_{m}^{-1}((\L,s,f)/\X) = \X\times_{[\P(1,m)/\G_{a}]}^{\nu}[\P(1,m)/\G_{a}]$ with respect to the universal Artin--Schreier cover $\wp_{m} : [\P(1,m)/\G_{a}]\rightarrow [\P(1,m)/\G_{a}]$. 

\begin{rem}
\label{rem:ASrootstackpts}
Proposition~\ref{prop:ASlinebundlesectionclassification} showed that for a scheme $X$, the $X$-points of $\P(1,m)$ are given by triples $(L,s,f)$ for a line bundle $L\rightarrow X$ and two non-simultaneously vanishing sections $s\in H^{0}(X,L)$ and $f\in H^{0}(X,L^{m})$. However, the classifying map used to define an Artin--Schreier root stack construction has target $[\P(1,m)/\G_{a}]$, so a global section of $L^{m}$ is sometimes more than what is necessary. In fact, the $X$-points $[\P(1,m)/\G_{a}](X)$ by definition are $\G_{a}$-torsors $P\rightarrow X$ together with $\G_{a}$-equivariant maps $P\rightarrow\P(1,m)$. Assuming $s$ is a {\it regular section} of $L$ (i.e. a nonzero divisor in each stalk), let $D = \div(s)$, so that $L = \orb_{X}(D)$, and consider the short exact sequence of sheaves
$$
\sesx{\orb_{X}}{\orb_{X}(mD)}{\orb_{X}(mD)|_{mD}}{s^{m}}{}. 
$$
This induces a long exact sequence in sheaf cohomology: 
$$
0\rightarrow H^{0}(X,\orb_{X})\rightarrow H^{0}(X,\orb_{X}(mD))\rightarrow H^{0}(mD,\orb_{X}(mD)|_{mD})\xrightarrow{\delta} H^{1}(X,\orb_{X})\rightarrow\cdots. 
$$
So one way to produce a $\G_{a}$-torsor on $X$ is to take the image under the connecting homomorphism $\delta$ of a section $f\in H^{0}(mD,\orb_{X}(mD)|_{mD})$. That is, $P = \delta(f)$ where $f$ is a section of \emph{the restriction of $L^{m}$ to the divisor} $mD$. One can show that this is bijective: the $X$-points of $[\P(1,m)/\G_{a}]$ are in one-to-one correspondence with triples $(L,s,f)$ with $(L,s)$ as usual and $f$ a section of $L^{m}$ supported on the divisor $mD$. 

Now we can give an explicit description of the points of $\wp_{m}^{-1}((L,s,f)/X)$ in the style of Remark~\ref{rem:rootstackpts}. If $X$ is a scheme, let $\V = X\times_{[\P(1,m)/\G_{a}]}[\P(1,m)/\G_{a}]$ be the actual pullback of the diagram in Definition~\ref{defn:ASrootstack}. Then for a test scheme $T$ the category $\V(T)$ consists of tuples $(T\xrightarrow{\varphi}X,M,t,g,\psi)$ where $M\rightarrow T$ is a line bundle with section $t\in H^{0}(T,M)$, $g\in H^{0}(mE,M^{m}|_{mE})$ (where $E = (t)$) and $\psi : M^{p}\xrightarrow{\sim}\varphi^{*}L$ is an isomorphism of line bundles such that 
$$
\psi(t^{p}) = \varphi^{*}s \quad\text{and}\quad \psi_{mpE}(g^{p} - t^{m(p - 1)}g) = \varphi_{mD}^{*}f
$$
where $(-)_{mpE}$ denotes the restriction to $mpE$ and likewise for $\varphi_{mD}^{*}$. By \cite[Prop.~A.7]{ab}, the $T$-points of $\wp_{m}^{-1}((L,s,f)/X) = \V^{\nu}$ has the same description when $T$ is a normal scheme. In general, the defining equations on $t$ and $g$ are more complicated. There is a similar description of the $T$-points of the Artin--Schreier root stack $\wp_{m}^{-1}((\L,s,f)/\X)$ when $\X$ is a stack. 

For our purposes, namely when $X$ (resp.~$\X$) is a curve (resp.~stacky curve), we will not need this level of control over the sections $f$. Indeed, \'{e}tale-locally, $H^{1}(X,\orb_{X}) = 0$ so any $f\in H^{0}(mD,L^{m}|_{mD})$ as above lifts to a section $F\in H^{0}(X,L^{m})$. Therefore, \'{e}tale-locally the $T$-points of $\wp_{m}^{-1}((L,s,f)/X)$ are given by $(\varphi,M,t,g,\psi)$ where $g\in H^{0}(T,M^{m})$ and the rest are as above (when $T$ is normal). While some of our results require global sections, the computations happen locally so this technical point is not a significant issue in the present article. 
\end{rem}

\begin{lem}
\label{lem:ASrootstacknatural}
For any morphism of stacks $h : \Y\rightarrow\X$ and line bundle $\L\rightarrow\X$ with sections $s$ of $\L$ and $f$ of $\L^{m}$, there is an isomorphism of algebraic stacks 
$$
\wp_{m}^{-1}((h^{*}\L,h^{*}s,h^{*}f)/\Y) \xrightarrow{\;\;\sim\;\;} \wp_{m}^{-1}((\L,s,f)/\X)\times_{\X}^{\nu}\Y. 
$$
\end{lem}

\begin{proof}
As before, this is an immediate consequence of the definition or the explicit description in Remark~\ref{rem:ASrootstackpts}. 
\end{proof}

\begin{ex}
\label{ex:ASrootstackP1}
Let $X = \P^{1} = \Proj k[x_{0},x_{1}]$ and suppose $Y\rightarrow X$ is the smooth projective model of the one-point cover given by the affine Artin--Schreier equation $y^{p} - y = x^{-m}$. Then $Y$ admits an additive $\Z/p\Z$-action such that $Y\rightarrow X$ is a Galois cover with group $\Z/p\Z$. There is an isomorphism of stacks 
$$
\wp_{m}^{-1}((\orb(1),x_{0},x_{1}^{m})/\P^{1}) \cong [Y/(\Z/p\Z)]
$$
which we describe now. As in Remark~\ref{rem:ASrootstackpts}, let $\V$ be the pullback $\P^{1}\times_{[\P(1,m)/\G_{a}]}[\P(1,m)/\G_{a}]$ so that $\wp_{m}^{-1}((\orb(1),x_{0},x_{1}^{m})/\P^{1}) = \V^{\nu}$, the normalized pullback. Also let $Y_{0}\rightarrow\P^{1}$ be the projective closure of the affine curve given by the equation $y^{p}x^{m} - yx^{m} = 1$, which is in general not normal. Then by \cite[Prop.~A.7]{ab}, it is enough to construct an isomorphism $\V(T)\cong [Y_{0}/(\Z/p\Z)](T)$ for any \emph{normal} test scheme $T$. By Remark~\ref{rem:ASrootstackpts}, assuming $T$ is ``local enough'', the category $\V(T)$ consists of tuples $(T\xrightarrow{\varphi}\P^{1},L,s,f,\psi)$ where $\psi : L^{p}\xrightarrow{\sim} \varphi^{*}\orb(1)$ is an isomorphism such that $\psi(s^{p}) = \varphi^{*}x_{0}$ and $\psi(f^{p} - fs^{m(p - 1)}) = \varphi^{*}x_{1}^{m}$. Define a mapping 
\begin{align*}
  \V(T) &\longrightarrow [Y/(\Z/p\Z)](T)\\
    (\varphi,L,s,f,\psi) &\longmapsto \left (\tikz[scale=1.2,baseline=12]{
  \node at (0,1) (a) {$\widetilde{L}$};
  \node at (1,1) (b) {$Y$};
  \node at (0,0) (c) {$T$};
  \draw[->] (a) -- (b);
  \draw[->] (a) -- (c);
}\right )
\end{align*}
where $\widetilde{L}\rightarrow T$ is the $\Z/p\Z$-bundle obtained by first constructing a $\G_{a}$-bundle $P\rightarrow T$ and showing its transition maps actually take values in $\Z/p\Z\subseteq\G_{a}$. 

Fix a cover $\{U_{i}\rightarrow T\}$ over which $L\rightarrow T$ is trivial; we may choose the $U_{i}$ small enough so that on each, either $s$ or $f$ is nonzero. Then $P$ can be constructed by specifying transition functions $\varphi_{ij} : U_{i}\cap U_{j}\rightarrow\G_{a}$ for any pair $U_{i},U_{j}$ in the cover. Let $s_{i} = s|_{U_{i}}$ and $f_{i} = f|_{U_{i}}$. There are three cases to consider. First, if $s$ is nonzero on both $U_{i}$ and $U_{j}$, then we let 
$$
\varphi_{ij} : t \longmapsto \frac{f_{i}(t)}{s_{i}(t)^{m}} - \frac{f_{j}(t)}{s_{j}(t)^{m}}\in\G_{a}. 
$$
If $s$ vanishes somewhere on $U_{i}$ and $U_{j}$, then set 
$$
\varphi_{ij} : t \longmapsto \frac{s_{i}(t)^{m}}{f_{i}(t)} - \frac{s_{j}(t)^{m}}{f_{j}(t)}. 
$$
Finally, if $s$ does not vanish on $U_{i}$ but does vanish somewhere on $U_{j}$, we let 
$$
\varphi_{ij} : t \longmapsto \frac{f_{i}(t)}{s_{i}(t)^{m}} - \frac{s_{j}(t)^{m}}{f_{j}(t)}. 
$$
It is easy to see that if $s$ is nonzero on three charts $U_{i},U_{j},U_{k}$, then the transition maps satisfy the additive cocycle relation $\varphi_{ij} + \varphi_{jk} + \varphi_{ki} = 0$. Similarly, for $s$ vanishing on any combination of $U_{i},U_{j},U_{k}$, the same cocycle relation holds, e.g.~if $s$ vanishes on $U_{i}$ but not on $U_{j}$ or $U_{k}$, then 
$$
\varphi_{ij} + \varphi_{jk} + \varphi_{ki} = \frac{s_{i}(t)^{m}}{f_{i}(t)} - \frac{f_{j}(t)}{s_{j}(t)^{m}} + \frac{f_{j}(t)}{s_{j}(t)^{m}} - \frac{f_{k}(t)}{s_{k}(t)^{m}} + \frac{f_{k}(t)}{s_{k}(t)^{m}} - \frac{s_{i}(t)^{m}}{f_{i}(t)} = 0. 
$$
Therefore $(\varphi_{ij})$ defines a $\G_{a}$-bundle $P\rightarrow T$ which is trivial over the cover $\{U_{i}\}$. One can show that the total space of $P$ is $(L\smallsetminus\{0\}\times L^{m})/\G_{m}$ over the locus where $s$ vanishes, while over the locus where $f$ vanishes, the total space is $(L\times L^{m}\smallsetminus\{0\})/\G_{m}$, where $\G_{m}$ acts on fibres with weights $(1,m)$ in both cases. Over any of the $U_{i}$, $P$ is trivialized by 
\begin{align*}
  \varphi_{i} : U_{i}\times\G_{a} &\longrightarrow P|_{U_{i}}\\
    (t,\alpha) &\longmapsto [s(t),f(t) + \alpha s(t)^{m}]. 
\end{align*}
Alternatively, $P$ can be defined as in Remark~\ref{rem:ASrootstackpts}. By construction, the transition maps are all $\Z/p\Z$-valued since $s$ and $f$ satisfy the equation $f^{p} - fs^{m(p - 1)} = \varphi^{*}x_{1}^{m}$, which over each type of trivialization looks like one of 
$$
\left (\frac{f}{s^{m}}\right )^{p} - \frac{f}{s^{m}} = \left (\frac{\varphi^{*}x_{1}}{s^{p}}\right )^{m} \qquad\text{or}\qquad \left (\frac{s^{m}}{f}\right )^{p} - \frac{s^{m}}{f} = -\frac{(s\varphi^{*}x_{1})^{m}}{f^{p + 1}}. 
$$
Thus $(\varphi_{ij})$ actually determine a $\Z/p\Z$-bundle $\widetilde{L}\rightarrow T$. One can also consider the Artin--Schreier sequence 
$$
\sesx{\Z/p\Z}{\G_{a}}{\G_{a}}{}{\wp}. 
$$
This induces an exact sequence 
$$
H^{1}(T,\Z/p\Z)\rightarrow H^{1}(T,\G_{a})\xrightarrow{\wp} H^{1}(T,\G_{a}). 
$$
Then $\wp(P)$ is the $\G_{a}$-bundle with transition functions $\wp(\varphi_{ij}) = \varphi_{ij}^{p} - \varphi_{ij}$, so using the defining equation for $s$ and $f$, one can again see that $\wp(\varphi_{ij})\in Z^{1}(T,\G_{a})$ is trivial in $H^{1}(T,\G_{a})$. Therefore by exactness, $P$ lifts to $\widetilde{L}\in H^{1}(T,\Z/p\Z)$. 

Next, there is a map $\widetilde{L}\rightarrow Y_{0}$ given by sending 
$$
(t,0)\in\widetilde{L}_{t} \longmapsto \left (\frac{s(t)}{\varphi^{*}x_{1}(t)},\frac{f(t)}{s(t)^{m}}\right )
$$
and extending by the $\Z/p\Z$-action on the fibre $\widetilde{L}_{t}$. That is, the entire fibre over $t$ is mapped to the Galois orbit of a corresponding point on $Y_{0}$. The resulting morphism 
$$
\V(T)\rightarrow [Y_{0}/(\Z/p\Z)](T)
$$
is an isomorphism for all $T$, and it is easy to check this is functorial in $T$. This proves that for any sufficiently local test scheme $T$ (in the sense of Remark~\ref{rem:ASrootstackpts}) which is normal, there is an isomorphism $\V(T)\xrightarrow{\sim}[Y_{0}/(\Z/p\Z)](T)$. The argument can be extended to all normal schemes $T$, so by \cite[Prop.~A.7]{ab}, we get an isomorphism of stacks $\wp_{m}^{-1}((\orb(1),x_{0},x_{1}^{m})/\P^{1})\cong [Y/(\Z/p\Z)]$. 

A similar analysis shows that for an arbitrary curve $X$ and for any $F\in k(X)\smallsetminus\wp(k(X))$, there is an isomorphism 
$$
\wp_{m}^{-1}((L,\sigma,\tau)/X)\xrightarrow{\;\sim\;} [Y_{F}/(\Z/p\Z)]
$$
where $(L,\sigma)$ corresponds to the divisor $\div(F)\in\Div(X)$, $\tau$ is a section of $L^{m}$ such that $\sigma$ restricts to a local parameter at any zero of $\tau$, $Y_{F}$ is the Galois cover of $X$ with birational Artin--Schreier equation $y^{p} - y = F(x)$. 
\end{ex}

\begin{ex}
When $m\equiv -1\pmod{p}$, we even have $\wp_{m}^{-1}((\orb(1),x_{0},x_{1}^{m})/\P^{1})\cong \V$, i.e.~the pullback of $\P^{1}$ along the universal Artin--Schreier cover $\wp_{m} : [\P(1,m)/\G_{a}]\rightarrow [\P(1,m)/\G_{a}]$ is already a normal stack. To show this, we need to write down an integral equation for $\wp_{m}^{-1}((\orb(1),x_{0},x_{1})/\P^{1})\times_{\P^{1}}T$ for any sufficiently local test scheme $T$ and show that it can be mapped to an integral equation for $Y\times_{\P^{1}}T\rightarrow\P_{T}^{1}$. In fact, the equation $f^{p} - fs^{m(p - 1)} = \varphi^{*}x_{1}^{m}$ on the level of sections can be written 
$$
\left (\frac{fs}{\varphi^{*}x_{1}^{n}}\right )^{p} - \left (\frac{fs}{\varphi^{*}x_{1}^{n}}\right )\left (\frac{s^{p}}{\varphi^{*}x_{1}}\right )^{n(p - 1)} = \frac{s^{p}}{\varphi^{*}x_{1}}
$$
where, as in the notation of Lemma~\ref{lem:ASinteqn}, $m + 1 = pn$. This pulls back to a similar equation on each $\wp_{m}^{-1}((\orb(1),x_{0},x_{1}^{m})/\P^{1})\times_{\P^{1}}T$ which then maps to $z^{p} - zx^{n(p - 1)} = x$ on $Y\times_{\P^{1}}T$ by sending 
$$
\left (\frac{fs}{\varphi^{*}x_{1}^{n}}\right ) \longmapsto z \quad\text{and}\quad \left (\frac{s^{p}}{\varphi^{*}x_{1}}\right ) \longmapsto x. 
$$
By Lemma~\ref{lem:ASinteqn}, $z^{p} - zx^{n(p - 1)} = x$ is an integral equation for $Y\times_{\P^{1}}T\rightarrow\P_{T}^{1}$ in this case since $m = pn - 1$. 
\end{ex}


In general, every Artin--Schreier root stack $\wp_{m}^{-1}((L,s,f)/X)$ can be covered in the \'{e}tale topology by ``elementary'' Artin--Schreier root stacks of the form $[Y/(\Z/p\Z)]$ as above -- this is completely analogous to the Kummer case described by Proposition~\ref{prop:coverbyrootstacks} but here is the formal statement. 

\begin{prop}
\label{prop:ASrootstacklocalquotient}
Let $\X = \wp_{m}^{-1}((L,s,f)/X)$ be an Artin--Schreier root stack of a scheme $X$ with jump $m$ along a triple $(L,s,f)$ and let $\pi : \X\rightarrow X$ be the coarse map. Then for any point $\bar{x} : \Spec k\rightarrow\X$, there is an \'{e}tale neighborhood $U$ of $x = \pi\circ\bar{x}$ such that $U\times_{X}\X\cong [Y/(\Z/p\Z)]$, where $Y$ is the Galois $\Z/p\Z$-cover of $U$ given by an affine Artin--Schreier equation $y^{p} - y = F(x)$. 
\end{prop}

Also, we have the following analogue of Theorem~\ref{thm:rootstackDM} for Artin--Schreier root stacks: 

\begin{thm}
\label{thm:ASrootstackDM}
If $\X$ is a Deligne--Mumford stack over a perfect field $k$ of characteristic $p > 0$, $m$ is an integer relatively prime to $p$ and $\L\rightarrow\X$ is a line bundle with sections $s$ of $\L$ and $f$ of $\L^{m}$, then $\wp_{m}^{-1}((\L,s,f)/\X)$ is a Deligne--Mumford stack. 
\end{thm}

\begin{proof}
Take an \'{e}tale map $p : U\rightarrow\X$ such that $p^{*}\L$ is trivial on the scheme $U$. It suffices to show $\wp_{m}^{-1}((\L,s,f)/\X)\times_{\X}U$ is Deligne--Mumford. The corresponding map $U\rightarrow\P(1,m)$ induced by $(p^{*}\L,p^{*}s,p^{*}f)$ lifts the composition $U\rightarrow \X\rightarrow [\P(1,m)/\G_{a}]$ along the quotient map $\P(1,m)\rightarrow [\P(1,m)/\G_{a}]$. Consider the composition $U\rightarrow\P(1,m)\rightarrow\P^{1}$. By Lemma~\ref{lem:ASrootstacknatural}, 
$$
\wp_{m}^{-1}((\L,s,f)/\X)\times_{\X}U \cong \wp_{m}^{-1}((p^{*}\L,p^{*}s,p^{*}f)/U) \cong \wp_{m}^{-1}((\orb(1),x_{0},x_{1}^{m})/\P^{1})\times_{\P^{1}}U. 
$$
Then Example~\ref{ex:ASrootstackP1} shows that $\wp_{m}^{-1}((\orb(1),x_{0},x_{1}^{m})/\P^{1}) \cong [Y/(\Z/p\Z)]$ where $Y$ is a smooth scheme and $\Z/p\Z$ acts on $Y\rightarrow\P^{1}$. Since $\Z/p\Z$ is an \'{e}tale group scheme, the quotient stack $[Y/(\Z/p\Z)]$ is Deligne--Mumford (cf.~\cite[Cor.~8.4.2]{ols}). Therefore $[Y/(\Z/p\Z)]\times_{\P^{1}}U$ is Deligne--Mumford, so $\wp_{m}^{-1}((\L,s,f)/\X)\times_{\X}U$ is Deligne--Mumford as required. 
\end{proof}


Next, we give a new characterization of $\Z/p\Z$-covers of curves in characteristic $p$ using Artin--Schreier root stacks. 

\begin{thm}
\label{thm:pcoverfactorsthroughASrootstack}
Let $k$ be an algebraically closed field of characteristic $p > 0$ and suppose $Y\rightarrow X$ is a finite separable Galois cover of curves over $k$ and $y\in Y$ is a ramification point with image $x\in X$ such that the inertia group $I(y\mid x)$ is $\Z/p\Z$. Then there are \'{e}tale neighborhoods $V\rightarrow Y$ of $y$ and $U\rightarrow X$ of $x$ such that $V\rightarrow U$ factors through an Artin--Schreier root stack 
$$
V\longrightarrow \wp_{m}^{-1}((L,s,f)/U) \rightarrow U
$$
for some $m$. 
\end{thm}

\begin{proof}
Let $\orb_{Y,y}$ and $\orb_{X,x}$ be the local rings at $y$ and $x$, respectively. Passing to completions, we may assume the extension $\widehat{\orb}_{Y,y}/\widehat{\orb}_{X,x}$ is isomorphic to $\widehat{A}/k[[t]]$ where $\widehat{A} = k[[t,u]]/(u^{p} - u - t^{m}g)$ for some $g\in k[[t]]$. Per an earlier discussion, after a change of formal coordinates, we may assume $g = 1$. By Artin approximation (\cite[Tag 0CAT]{sp}), the extension of henselian rings $\orb_{Y,y}^{h}/\orb_{X,x}^{h}$ also has this form, but since the henselization of a local ring is a direct limit over the finite \'{e}tale neighborhoods of the ring, there must be finite \'{e}tale neighborhoods $V_{0} = \Spec B\rightarrow Y$ of $y$ and $U_{0} = \Spec A\rightarrow X$ of $x$ such that the corresponding ring extension is of the form $B/A = B/k[t]$, where $B = k[t,u]/(u^{p} - u - t^{m})$. Furthermore, by results in section 2 of \cite{harb}, there is an \'{e}tale neighborhood $U$ of $x$ such that:
\begin{enumerate}[\quad (i)]
  \item $U_{0} = U\smallsetminus\{x\}$ is an affine curve with local coordinate $t$, and
  \item the cover $V := U\times_{X}Y\rightarrow U$ is isomorphic to the one-point Galois cover of $U$ defined by the birational Artin--Schreier equation $u^{p} - u = t^{m}$ over $U_{0}$. 
\end{enumerate}
Then $\Z/p\Z$ acts (as an \'{e}tale group scheme) on $V$ via the usual action on this Artin--Schreier extension and by Example~\ref{ex:ASrootstackP1}, $[V/(\Z/p\Z)] \cong \wp_{m}^{-1}((L,s,f)/U)$ for some $(L,s,f)$, so it follows immediately that $V\rightarrow U$ factors through this Artin--Schreier root stack. 
\end{proof}

\begin{rem}
An alternate proof of Theorem~\ref{thm:pcoverfactorsthroughASrootstack} in the spirit of \cite{gar} goes as follows. Since the morphism $\pi : Y\rightarrow X$ has abelian inertia group at $x$, by geometric class field theory (cf.~\cite[Ch.~IV, Sec.~2, Prop.~9]{ser}) there is a modulus $\frak{m}$ on $X$ with support contained in the branch locus of $\pi$ and a rational map $\varphi : X\rightarrow J_{\frak{m}}$ to the generalized Jacobian of $X$ with respect to this modulus, such that $Y\cong \varphi^{*}J'$ for a cyclic isogeny $J'\rightarrow J_{\frak{m}}$ of degree $p$. Meanwhile, every such isogeny $J'\rightarrow J_{\frak{m}}$ is a pullback of the Artin--Schreier isogeny $\wp : \G_{a}\rightarrow\G_{a},\alpha\mapsto\alpha^{p} - \alpha$: 
\begin{center}
\begin{tikzpicture}[scale=2]
  \node at (0,1) (a) {$J'$};
  \node at (1,1) (b) {$\G_{a}$};
  \node at (2,1) (c) {$\P(1,m)$};
  \node at (0,0) (d) {$J_{\frak{m}}$};
  \node at (1,0) (e) {$\G_{a}$};
  \node at (2,0) (f) {$\P(1,m)$};
  \draw[->] (a) -- (b);
  \draw[right hook ->] (b) -- (c);
  \draw[->] (a) -- (d);
  \draw[->] (b) -- (e) node[right,pos=.5] {$\wp$};
  \draw[->] (c) -- (f) node[right,pos=.5] {$\wp_{m}$};
  \draw[->] (d) -- (e);
  \draw[right hook ->] (e) -- (f);
\end{tikzpicture}
\end{center}
Let $U'$ be an \'{e}tale neighborhood of $X$ on which $\varphi$ is defined and set $U = U'\cup\{x\}$. Then over $U$, $\varphi$ can be extended to a morphism $\overline{\varphi} : U\rightarrow\P(1,m)$ such that $x$ maps to the stacky point at infinity, where $m = \ord_{x}\frak{m} - 1$ (This $m$ is also equal to $\cond_{Y/X}(x) - 1$, the conductor of the extension minus $1$.) By Proposition~\ref{prop:ASlinebundlesectionclassification}, $\overline{\varphi}$ may also be defined by specifying a triple $(L,s,f)$ on $U$: the pair $(L,s)$ corresponds to the effective divisor $D = x$ and $f$ is a section of $\orb(mD) = L^{\otimes m}$ determined from the affine equation expressing the pullback of the isogeny $J'\rightarrow J_{\frak{m}}$ to $U$. 

Let $V = \pi^{*}U$ which is an \'{e}tale neighborhood of $Y$. Pulling back $(L,s,f)$ to $V$ locally defines $\overline{\psi}$ in the following diagram: 
\begin{center}
\begin{tikzpicture}[xscale=4,yscale=2]
  \node at (0,1) (a) {$\wp_{m}^{-1}((L,s,f)/U)$};
  \node at (1,1) (b) {$[\P(1,m)/\G_{a}]$};
  \node at (0,0) (c) {$U$};
  \node at (1,0) (d) {$[\P(1,m)/\G_{a}]$};
  \node at (0,2) (a1) {$V$};
  \draw[->] (a) -- (b);
  \draw[->] (a) -- (c);
  \draw[->] (b) -- (d) node[right,pos=.5] {$\wp_{m}$};
  \draw[->] (c) -- (d) node[above,pos=.5] {$\overline{\varphi}$};
  \draw[->,dashed] (a1) -- (a);
  \draw[->] (a1) to[out=225,in=135] (c);
    \node at (-.5,1) {$\pi$};
  \draw[->] (a1) -- (b) node[above,pos=.5] {$\overline{\psi}$};
\end{tikzpicture}
\end{center}
By the universal property of $\wp_{m}^{-1}((L,s,f)/U)$, there is a map $V\rightarrow\wp_{m}^{-1}((L,s,f)/U)$ factoring $\pi$ locally as required. 
\end{rem}

\begin{thm}
\label{thm:locallyASrootstack}
Let $\X$ be a stacky curve over a perfect field $k$ of characteristic $p > 0$. Then 
\begin{enumerate}[\quad (1)]
  \item If $\X$ contains a stacky point $x$ of order $p$, there is an open substack $\U\subseteq\X$ containing $x$ such that $\U\cong\wp_{m}^{-1}((L,s,f)/U)$ where $(m,p) = 1$, $U$ is an open subscheme of the coarse space $X$ of $\X$ and $(L,s,f)\in\DIV^{[1,m]}(U)$. 
  \item Suppose all the nontrivial stabilizers of $\X$ are cyclic of order $p$. If $\X$ has coarse space $\P^{1}$, then $\X$ is isomorphic to a fibre product of Artin--Schreier root stacks of the form $\wp_{m}^{-1}((L,s,f)/\P^{1})$ for $(m,p) = 1$ and $(L,s,f)\in\DIV^{[1,m]}(\P^{1})$. 
\end{enumerate}
\end{thm}

\begin{proof}
(1) Let $m$ be the unique positive integer such that for any \'{e}tale presentation $Y\rightarrow\X$ and any point $y\in Y$ mapping to $x$, the ramification jump of the induced cover $Y\rightarrow X$ at $y$ is $m$. Let $U\subseteq X$ be a subscheme such that $x$ is the only stacky point of $\U := \pi^{-1}(U)\subseteq\X$, where $\pi : \X\rightarrow X$ is the coarse map. Let $L = \orb_{U}(x)$. Then for any \'{e}tale map $\varphi : T\rightarrow\U$, there is a canonical Artin--Schreier root of the line bundle $\varphi^{*}\pi^{*}L$. Indeed, since $\pi\circ\varphi : T\rightarrow U$ is a one-point cover of curves with inertia group $\Z/p\Z$ and ramification jump $m$ at $x$, Theorem~\ref{thm:pcoverfactorsthroughASrootstack} says there are sections $s$ and $f$ such that $\pi\circ\varphi$ factors as $T\rightarrow\wp_{m}^{-1}((L,s,f)/U)\rightarrow U$. On the other hand, by the universal property of $\wp_{m}^{-1}((L,s,f)/U)$, there is a canonical morphism $\U\rightarrow\wp_{m}^{-1}((L,s,f)/U)$, which is therefore an isomorphism. 

(2) Let $B = \{x_{1},\ldots,x_{r}\}$ be the finite set of points in $\P^{1}$ covered by points in $\X$ with nontrivial automorphism groups. Since $k(\P^{1}) = k(t)$, we can choose $F_{i}\in k(t)$ having a pole of any desired order at $x_{i}$ for each $1\leq i\leq r$. For instance, if $m_{i}$ is the ramification jump at $x_{i}$ as defined in (1), then we can arrange for $\ord_{x_{i}}(F_{i}) = -m_{i}$ for each $i$. Again by section 2 of \cite{harb}, there is a proper curve $Y_{i}\rightarrow\P^{1}$ with affine Artin--Schreier equation $y^{p} - y = F_{i}(t)$, corresponding to the function field $k[t,y]/(y^{p} - y - F_{i})$ as an extension of $k(t)$. We claim 
$$
\X \cong \Y := [Y_{1}/(\Z/p\Z)]\times_{\P^{1}}\cdots\times_{\P^{1}}[Y_{r}/(\Z/p\Z)]
$$
which will prove (2) after applying Example~\ref{ex:ASrootstackP1} to each $[Y_{i}/(\Z/p\Z)]$. We construct a map $\X\rightarrow\Y$ as follows. Let $T$ be an arbitrary $k$-scheme. Since both coarse maps $\X\rightarrow\P^{1}$ and $\Y\rightarrow\P^{1}$ are isomorphisms away from $B$, we only need to specify the image of each stacky point $Q\in\X(T)$. Note that $\pi(Q)\in B$, so $\pi(Q) = x_{i}$ for some $1\leq i\leq r$. If $T = \Spec k$, then such a $Q$ is represented by a gerbe $B(\Z/p\Z)\hookrightarrow\X$, so send $Q\in\X(k)$ to the point of $[Y_{i}/(\Z/p\Z)](k)$ corresponding to 
\begin{center}
\begin{tikzpicture}[xscale=1.7,yscale=1.2]
  \node at (0,1) (a) {$\Spec k$};
  \node at (1,1) (b) {$Y_{i}$};
  \node at (0,0) (c) {$B(\Z/p\Z)$};
  \draw[->] (a) -- (b) node[above,pos=.5] {$g_{Q}$};
  \draw[->] (a) -- (c);
\end{tikzpicture}
\end{center}
where $\Spec k\rightarrow B(\Z/p\Z)$ is the universal $\Z/p\Z$-bundle and $g_{Q}$ has image $x_{i}\in Y_{i}$. Extending this to any scheme $T$ is easy: replace the above diagram with 
\begin{center}
\begin{tikzpicture}[xscale=1.7,yscale=1.2]
  \node at (0,1) (a) {$T$};
  \node at (1,1) (b) {$Y_{i}\times_{k}T$};
  \node at (0,0) (c) {$B(\Z/p\Z)\times_{k}T$};
  \draw[->] (a) -- (b) node[above,pos=.5] {$g_{Q}$};
  \draw[->] (a) -- (c);
\end{tikzpicture}
\end{center}
and define $g_{Q}$ by $t\mapsto (x_{i},t)$. This defines an equivalence of categories $\X(T)\rightarrow\Y(T)$ for any scheme $T$, so we can invoke Lemma~\ref{lem:CFGequiv} to ensure that these fibrewise equivalences assemble into an equivalence of stacks $\X\xrightarrow{\sim}\Y$. 
\end{proof}

\begin{rem}
As illustrated in in the proof of Theorem~\ref{thm:locallyASrootstack}, the ramification data at a stacky $\Z/p\Z$-point $x\in\X$ (i.e.~the ramification jump $m$) can be determined from the stack itself, namely, by which \'{e}tale covers are allowed at $x$. This is unique to the positive characteristic case; in characteristic $0$, one only needs to know the order of the stabilizer at every stacky point to understand the entire stacky structure. In contrast, for each fixed $m$ there is a \emph{family} of nonisomorphic stacky curves, with coarse space $\P^{1}$ and order $p$ stabilizers at any prescribed points, parametrized by the possible choices of $f$. 
\end{rem}

\begin{rem}
The local structure of a stacky curve in characteristic $p$ is separable since the stack is by definition Deligne--Mumford and generically a scheme. Thus around a wildly ramified point, one does not have structures like $[U/\mu_{p}]$ or $[U/\alpha_{p}]$ which would be more problematic, but will be interesting to have a description of in the future. 
\end{rem}

\begin{ex}
\label{ex:nonP1cex}
When the coarse space is not $\P^{1}$, Theorem~\ref{thm:locallyASrootstack}(2) is false. For example, let $k$ be an algebraically closed field of characteristic $p > 0$ and let $E$ be an elliptic curve over $k$ with a rational point $P$. By Riemann--Roch, $h^{0}(E,\orb(P)) = 1$ so every global section of $\orb(P)$ vanishes at $P$ and we can exploit this limitation to construct a counterexample. Let $\E$ be a stacky curve with coarse space $E$ and a stacky point of order $p$ at $P$ with ramification jump $1$. This can be achieved, by Theorem~\ref{thm:locallyASrootstack}(1), through taking a local Artin--Schreier root stack $\wp_{1}^{-1}((\orb(P),s_{P},f_{P})/U)$ where $U$ is an \'{e}tale neighborhood of $P$, $(\orb(P),s_{P})$ is the pair corresponding to the effective divisor $P$ (in $U_{P}$) and $f_{P}$ is a section of $\orb(P)|_{U_{P}}$ not vanishing at $P$. Then if $\E$ were a global Artin--Schreier root stack over $E$, there would be some $s\in H^{0}(E,\orb(P))$ vanishing at $P$ and some $f\in H^{0}(E,\orb(P))$ not vanishing at $P$ such that $\E \cong \wp_{1}^{-1}((\orb(P),s,f)/E)$ but such an $f$ does not exist by the reasoning above. Variants of this counterexample can be constructed for single stacky points with ramification jump $m > 1$, as well as multiple stacky points with the same ramification jump which cannot be obtained through a fibre product of global Artin--Schreier root stacks. This argument also works for higher genus curves, so the $\P^{1}$ case is quite special. 
\end{ex}


\section{Canonical Rings}
\label{sec:canrings}

Let $X$ be a compact complex manifold and $\Omega_{X} = \Omega_{X(\C)}$ the sheaf of holomorphic differential $1$-forms on $X$. Then in many situations $X$ can be outfitted with the structure of a complex projective \emph{variety} by explicitly embedding it into projective space using different tensor powers of $\Omega_{X}$. For example, if $X$ is a Riemann surface of genus $g\geq 2$ that is not hyperelliptic, the canonical map $X\rightarrow |\orb_{X}| \cong \P^{g - 1}$ is an embedding. More generally, any Riemann surface $X$ of genus $g\geq 2$ can be recovered as $\Proj R(X)$ where 
$$
R(X) := \bigoplus_{k = 0}^{\infty} H^{0}(X,\Omega_{X}^{k})
$$
is the {\it canonical ring} of $X$. In general, the isomorphism $X\cong\Proj R(X)$ may not hold since $\Omega_{X}$ need not be ample, but it is still a useful gadget for studying the algebraic properties of $X$. 

In \cite{vzb}, the authors extend this strategy to tame stacky curves by giving explicit generators and relations for the canonical ring 
$$
R(\X) = \bigoplus_{r = 0}^{\infty} H^{0}(\X,\Omega_{\X}^{r})
$$
of a stacky curve $\X$, where $\Omega_{\X} = \Omega_{\X/k}$ is the sheaf of differentials defined in Section~\ref{sec:divisorsbundles}. This has numerous applications, perhaps most importantly to the computation of the ring of modular forms for a given congruence subgroup of $SL_{2}(\Z)$ (see Section~\ref{sec:future}), and holds in all characteristics provided the stack $\X$ has no wild ramification. However, numerous (stacky) modular curves one would like to study in characteristic $p$, such as $X_{0}(N)$ with $p\mid N$, have wild ramification and therefore fall outside the scope of results like \cite[Thm.~1.4.1]{vzb}. In this section, we will show how the Artin--Schreier root stack construction can shed light on the local structure of such curves in characteristic $p$ and how this can be used to study canonical rings. First, we give a generalization of the Riemann--Hurwitz formula (Proposition~\ref{prop:stackyRH}) for stacky curves with arbitrary (finite) stabilizers. The upshot is that we can then compute the canonical divisor from knowledge of the canonical divisor on the coarse space and the ramification filtrations at the stacky points of $\X$. 

Let $\X$ be a stacky curve with coarse moduli space $X$ and suppose $x\in\X(k)$ is a wild stacky point, i.e.~a stacky point with stabilizer $G_{x}$ such that $p$ divides $|G_{x}|$. Then by Lemma~\ref{lem:locallyquotientstack}, $\X$ is locally given by a quotient stack of the form $[U/G_{x}]$ where $U$ is a scheme. Consider the composite $U\rightarrow [U/G_{x}]\xrightarrow{\pi} X$ where $\pi : \X\rightarrow X$ is the coarse map. Let $W$ denote the schematic image of $U$ in $X$, so that $U\rightarrow W$ is a one-point cover of curves with Galois group $G_{x}$. Then $G_{x}$ has a higher ramification filtration (for the lower numbering), say $(G_{x,i})_{i\geq 0}$. It is easy to check this filtration does not depend on the choice of $U$, so we call $(G_{x,i})_{i\geq 0}$ the {\it higher ramification filtration at the stacky point} $x$ (for the lower numbering). In the tame case, it is common to put $G_{x,0} = G_{x}$ and $G_{x,i} = 1$ for $i > 0$, so the Riemann--Hurwitz formula in the tame case is subsumed by the following formula. 

\begin{prop}[Stacky Riemann--Hurwitz]
\label{prop:wildstackyRH}
For a stacky curve $\X$ with coarse moduli space $\pi : \X\rightarrow X$, the formula 
$$
K_{\X} = \pi^{*}K_{X} + \sum_{x\in\X(k)}\sum_{i = 0}^{\infty} (|G_{x,i}| - 1)x
$$
defines a canonical divisor $K_{\X}$ on $\X$. 
\end{prop}

\begin{proof}
As before, begin by assuming $\X$ has a single stacky point $P$. The tame case was proven in Proposition~\ref{prop:stackyRH} but this proof subsumes it, so in theory $P$ could have any stabilizer group $G$. Lemma~\ref{lem:locallyquotientstack} says that $\X\cong [U/G]$ for a scheme $U$. If $f : U\rightarrow [U/G]$ is the quotient map, then $f^{*}\Omega_{\X/X} \cong \Omega_{U/X}$ so the stalk of $\Omega_{\X/X}$ at $P$ has length equal to the different of the local ring extension $\widehat{\orb}_{\X,P}/\widehat{\orb}_{X,\pi(P)}$. This is precisely $\sum_{i = 0}^{\infty} (|G_{P,i}| - 1)$, by \cite[Ch.~IV, Prop.~4]{ser} for example, so the formula defines a canonical divisor on $[U/G]$. The local argument extends to the general case once again because $\X\rightarrow X$ is an isomorphism away from the stacky locus. 
\end{proof}

\begin{ex}
\label{ex:AScanonicaldiv}
Consider the Artin--Schreier cover $Y\rightarrow\P^{1}$ of Example~\ref{ex:ASrootstackP1} defined birationally by the Artin--Schreier equation $y^{p} - y = f(x)$, where $f$ is a degree $m$ polynomial. The resulting quotient stack $\X = [Y/(\Z/p\Z)]$ has a single stacky point $Q$ lying above $\infty$. Let $(G_{i})_{i\geq 0}$ be the higher ramification filtration of the inertia group at $Q$. Then the coarse space of $\X$ is $Y/(\Z/p\Z) = \P^{1}$, with coarse map $\pi : \X\rightarrow\P^{1}$, and 
$$
\pi^{*}K_{\P^{1}} + \sum_{i = 0}^{\infty} (|G_{i}| - 1)Q = -2H + \sum_{i = 0}^{m} (p - 1)Q = -2H + (m + 1)(p - 1)Q 
$$
where $H$ is a hyperplane, i.e.~a point, in $\P^{1}$ which we take to be distinct from $\infty$. Thus we can take $K_{\X} = -2H + (m + 1)(p - 1)Q$. More generally, if $f$ is a rational function with poles at $x_{1},\ldots,x_{r}$ of orders $m_{1},\ldots,m_{r}$, respectively, then 
$$
K_{\X} = -2H + (m + 1)(p - 1)Q + \sum_{j = 1}^{r} (m_{j} + 1)(p - 1)Q_{j}
$$
where $Q_{j}$ is the stacky point lying above $x_{j}$, $Q$ is again the stacky point above $\infty$ and $m = \deg(f)$. 
\end{ex}


As in the tame case, we may define the {\it Euler characteristic} $\chi(\X) = -\deg(K_{\X})$ and the {\it genus} $g(\X)$ via $\chi(\X) = 2 - 2g(\X)$. 

\begin{cor}
For a stacky curve $\X$ over an algebraically closed field $k$, with coarse space $X$, 
$$
g(\X) = g(X) + \frac{1}{2}\sum_{x\in\X(k)}\sum_{i = 0}^{\infty} \left (\frac{1}{[G_{x} : G_{x,i}]} - \frac{1}{|G_{x}|}\right )
$$
where $\pi : \X\rightarrow X$ is the coarse map. 
\end{cor}

\begin{ex}
If $Y\rightarrow\P^{1}$ is the Artin--Schreier cover given by $y^{p} - y = f(x)$ with $f\in k[x]$ of degree $(m,p) = 1$ and $\X = [Y/(\Z/p\Z)]$, then 
$$
g(\X) = \frac{(m + 1)(p - 1)}{2p}
$$
where $m = -\ord_{\infty}(f)$. For instance, when $m = 1$, $\X = [\P^{1}/(\Z/p\Z)]$ has canonical divisor $K_{\X} = -2H + 2(p - 1)Q$ and genus $g(\X) = \frac{p - 1}{p} = 1 - \frac{1}{p}$, similar to the tame case, where $[\P^{1}/\mu_{r}]$ has genus $1 - \frac{1}{r}$. However, if $m > 0$, the genus formula again illustrates that we have infinitely many non-isomorphic stacky $\P^{1}$'s. 
\end{ex}

\begin{ex}
Let $\char k = 3$ and let $E\rightarrow\P^{1}$ be the Artin--Schreier cover defined by the equation $y^{2} = x^{3} - x$. In general, the genus of an Artin--Schreier curve in characteristic $p$ with jump $m$ is $\frac{(p - 1)(m - 1)}{2}$ so in this case we have $g(E) = 1$. Thus $E$ is an elliptic curve and it is well-known (and easy to check) that $\omega = \frac{dx}{2y}$ is a differential form on $E$. Since $\dim H^{0}(E,\Omega_{E}^{1}) = g(E) = 1$, $\omega$ is, up to multiplication by an element of $k^{\times}$, the only nonzero, holomorphic differential form on $E$. 

Meanwhile, the stack $\X = [E/(\Z/3\Z)]$ also has genus $g(\X) = \frac{(3 - 1)(2 + 1)}{6} = 1$, so we might call it a ``stacky elliptic curve''. Notice that the group action of $\Z/3\Z$ on $E$, which is induced by $x\mapsto x + 1$, leaves $\omega$ invariant. Therefore $\omega$ generates the vector space of differential forms on $\X$, which is equivalently the space of $\Z/3\Z$-invariant differential forms on $E$. However, the coarse space here is $\P^{1}$ which has $H^{0}(\P^{1},\Omega_{\P^{1}}) = 0$. 
\end{ex}

In general, if $\X$ has coarse space $X$ and admits a presentation by a scheme $f : Y\rightarrow\X$, then these three cohomology groups, $H^{0}(Y,f^{*}\Omega_{\X}),H^{0}(\X,\Omega_{\X})$ and $H^{0}(X,\Omega_{X})$, need not be the same. The genus may even \emph{increase} along $f$, although this is already true in the characteristic $0$ case. 

Once we have our hands on the canonical divisor, the next step in trying to compute the canonical ring of a stacky curve is to apply a suitable version of Riemann--Roch to $K_{\X}$ and count dimensions. When $\X$ is a tame stacky curve, we do this as follows. For a divisor $D$ on $\X$, Lemma~\ref{lem:floor} implies that $H^{0}(\X,\orb_{\X}(D)) \cong H^{0}(X,\orb_{X}(\lfloor D\rfloor)$ where $\lfloor D\rfloor$ denotes the floor divisor. Then at any tame stacky point $x$, the stabilizer group $G_{x}$ is cyclic and by the tame Riemann--Hurwitz formula (Proposition~\ref{prop:stackyRH}), the coefficient of the canonical divisor at $x$ is $\frac{|G_{x}| - 1}{|G_{x}|}$. As a consequence, for any divisor $D$ on $\X$, $\lfloor K_{\X} - D\rfloor = K_{X} - \lfloor D\rfloor$. This yields the following version of Riemann--Roch for $\X$: 
\begin{align*}
  h^{0}(\X,D) - h^{0}(\X,K_{\X} - D) &= h^{0}(X,\lfloor D\rfloor) - h^{0}(X,\lfloor K_{\X} - D\rfloor)\\
    &= h^{0}(X,\lfloor D\rfloor) - h^{0}(X,K_{X} - \lfloor D\rfloor)\\
    &= \deg(\lfloor D\rfloor) - g(X) + 1. 
\end{align*}
(This appears as Corollary 1.189 in \cite{beh}, for example.) 

\begin{ex}
\label{ex:wtdprojlinecanring}
Let $a$ and $b$ be relatively prime integers that are not divisible by $\char k$ and consider the weighted projective line $\X = \P(a,b)$. Then $\X$ is a stacky $\P^{1}$ with two stacky points $P$ and $Q$ having cyclic stabilizers $\Z/a\Z$ and $\Z/b\Z$, respectively. Thus $\lfloor K_{\X}\rfloor = K_{\P^{1}} = -2H$ so $h^{0}(\X,rK_{\X}) = 0$ for all $r\geq 1$. In this case the canonical ring is trivial: $R(\X) = H^{0}(\X,\orb_{\X})\cong k$. This agrees with Example 5.6.9 in \cite{vzb}: $\P(a,b)$ is {\it hyperbolic}, i.e.~$\deg K_{\X} < 0$, so $R(\X)\cong k$. 
\end{ex}

\begin{ex}
Assume $\char k\not = 2$. Let $\X$ be a stacky curve with a single stacky point $Q$ of order $2$ and with coarse space $X$ of genus $1$. Then $K_{X} = 0$ so $K_{\X} = \frac{1}{2}Q$. Thus by Riemann--Roch, for any $n\geq 1$ we have 
$$
h^{0}(\X,nK_{\X}) = \begin{cases}
  1, & n = 0,1\\
  \left\lfloor\frac{n}{2}\right\rfloor, & n\geq 2. 
\end{cases}
$$
Example 5.7.1 in \cite{vzb} gives an explicit description of $R(\X)$ in terms of generators and relations, but for now, the dimension count is the essential information. 
\end{ex}

The subtle point above is that for a divisor $D$ on a stacky curve $\X$ with coarse space $X$, $\lfloor K_{\X} - D\rfloor = K_{X} - \lfloor D\rfloor$ need only hold when $\X$ is \emph{tame}. We have seen that this is not the case in the wild case. For example the stack $\X = [Y/(\Z/p\Z)]$ from Example~\ref{ex:AScanonicaldiv} has canonical divisor $K_{\X} = -2H + (m + 1)(p - 1)Q$, so $\lfloor K_{\X}\rfloor \not = K_{\P^{1}} = -2H$ for most choices of $m$. However, we can still apply Riemann--Roch to obtain new information in the wild case. 

\begin{ex}
Already for a wild $\P^{1}$ in characteristic $p$ we have new behavior compared to the tame case (see Example~\ref{ex:wtdprojlinecanring}). Let $\X$ be the quotient stack $[Y/(\Z/p\Z)]$ given by affine Artin--Schreier equation $y^{p} - y = \frac{1}{x^{m}}$, as in Example~\ref{ex:AScanonicaldiv}. We computed $K_{\X} = -2H + (m + 1)(p - 1)Q$, where $Q$ is the single stacky point over $\infty$ of order $p$. Then by Lemma~\ref{lem:floor} and Riemann--Roch, 
$$
h^{0}(\X,nK_{\X}) = h^{0}(\P^{1},\lfloor nK_{\X}\rfloor) = \deg(\lfloor nK_{\X}\rfloor) + 1 + h^{0}(\P^{1},K_{\P^{1}} - \lfloor nK_{\X}\rfloor). 
$$
For $n = 1$, this is already a new formula: 
$$
\lfloor K_{\X}\rfloor = -2H + \left (m - \left\lfloor\frac{m}{p}\right\rfloor\right )\infty. 
$$
Therefore $\deg(\lfloor K_{\X}\rfloor) = m - \left\lfloor\frac{m}{p}\right\rfloor - 2$. This also shows that $K_{\P^{1}} - \lfloor K_{\X}\rfloor$ is non-effective when $m\geq 2$, so we get 
$$
h^{0}(\X,K_{\X}) = \max\left\{m - \left\lfloor\frac{m}{p}\right\rfloor - 1,1\right\}. 
$$
More generally, since $\left\lfloor\frac{k(p - 1)}{p}\right\rfloor = k - \left\lfloor\frac{k}{p}\right\rfloor - 1$, we can compute 
$$
\lfloor nK_{\X}\rfloor = -2nH + \left (n(m + 1) - \left\lfloor\frac{n(m + 1)}{p}\right\rfloor - 1\right )\infty. 
$$
So $\deg(\lfloor nK_{\X}\rfloor) = n(m - 1) - \left\lfloor\frac{n(m + 1)}{p}\right\rfloor - 1$ and again, when $m\geq 2$, $K_{X} - \lfloor K_{\X}\rfloor$ is non-effective. By Riemann--Roch, 
$$
h^{0}(\X,nK_{\X}) = \max\left\{n(m - 1) - \left\lfloor\frac{n(m + 1)}{p}\right\rfloor,1\right\}. 
$$
\end{ex}


\section{Future Directions}
\label{sec:future}

To close, we describe one application of the computations in Section~\ref{sec:canrings} and preview the author's work in progress generalizing the Artin--Schreier root stack to moduli of $p$th power roots of line bundles and sections.

\subsection{Modular Forms in Characteristic $p$}

Modular forms are ubiquitous objects in modern mathematics, appearing in a startling number of places, such as: the study of the Riemann zeta function and more general $L$-functions; the proof of the Modularity Theorem and related results by Wiles, Taylor, et al.; the representation theory of finite groups; sphere packing problems; and quantum gravity in theoretical physics. They are a critical tool in the Langlands Program and in the study of the Birch and Swinnerton-Dyer Conjecture. For an overview of the modern theory, see \cite{BvdGHZ}, and for a longer survey, see \cite{ono}. A useful feature of modular forms is that they fall into finite dimensional vector spaces and therefore possess linear relations among their coefficients that encode a wealth of number theoretic data. 

Recall that for an even integer $k$, a classical {\it modular form} of weight $k$ is a holomorphic function on the complex upper half-plane $f : \h\rightarrow\C$ satisfying a holomorphicity condition at the point ``$\infty$'' (thought of as $i\infty$) and a modular condition: 
$$
f(z) = (cz + d)^{-k}f(gz) \quad\text{for any } g = \begin{pmatrix} a & b \\ c & d\end{pmatrix}\in\Gamma
$$
where $\Gamma$ is a congruence subgroup of $SL_{2}(\Z)$ acting on $z\in\h$ by fractional linear transformations: 
$$
\begin{pmatrix} a & b \\ c & d\end{pmatrix}\cdot z = \frac{az + b}{cz + d}. 
$$
The vector space of modular forms of weight $k$, denoted $\M_{k}(\Gamma)$, can be computed in principle for any $\Gamma$ and any even $k\in\Z$, but the formulas themselves conceal a rich algebro-geometric structure that underpins these vector spaces. 

To explain this algebro-geometric structure, we recall the following alternate definition of modular forms. Let $Y(\Gamma) = \h/\Gamma$ be the orbit space of the action of $\Gamma$ on the upper half-plane. Then $Y(\Gamma)$ admits a complex structure making it into an open Riemann surface. One may compactify this Riemann surface by adding orbits corresponding to the cusps of $\Gamma$ to obtain a proper Riemann surface denoted $X(\Gamma)$. Then a modular form $f$ of weight $k$ determines a holomorphic differential $k/2$-form $f(z)\, dz^{k/2}$ on $X(\Gamma)$. More precisely, $X(\Gamma)$ admits the structure of a complex orbifold, or equivalently, a complex stacky curve when we add the set of cusps to $\h$ and take the quotient stack $[\h\cup\Gamma\{\infty\}/\Gamma]$. Then the image of each orbit of elliptic points is a stacky point of order $n$, where $n$ is the order of the isotropy subgroup at any elliptic point in that orbit. 

\begin{prop}
Let $D$ be the $\Q$-divisor of cusps on $X = X(\Gamma)$. Then 
$$
\M_{k}(\Gamma) \cong H^{0}(X,\Omega_{X/\C}^{k/2}(D))
$$
for every even integer $k$, where $\Omega_{X/\C}^{1}$ is the canonical sheaf on $X$. 
\end{prop}

In particular, the results of \cite{vzb} on canonical rings of stacky curves apply here to give an explicit (and in some cases new) description of the ring of modular forms $\M_{*}(\Gamma) := \bigoplus_{k\in 2\Z} \M_{k}(\Gamma)$ for any congruence subgroup $\Gamma$. 

Over an arbitrary field $K$, one does not have access to the complex structure of a modular curve like in the complex case. Nevertheless, a suitable notion of modular forms can be defined over $K$. Following \cite{kat}, one defines a {\it geometric modular form} of weight $k$ over $K$ to be the assignment of an element $f(E/A,\omega)\in A$ to every $K$-algebra $A$, elliptic curve $E/K$ and nonzero differential $\omega\in H^{0}(E,\Omega_{E/A}^{1})$ satisfying a modularity condition 
$$
f(E/A,\alpha\omega) = \alpha^{-k}f(E/A,\omega) \quad\text{for all } \alpha\in A^{\times},
$$
and naturality conditions with respect to maps of $K$-algebras $A\rightarrow B$ and elliptic curves $E\rightarrow E'$, as well as an analogue of the holomorphicity condition that one can impose formally using the Tate curve (cf.~\cite[Appendix 1.1]{kat}). As in the classical case, geometric modular forms of level $N$ can be specified by imposing a torsion structure on elliptic curves over $K$. Different choices of torsion structure lead to the construction of the modular curves $X_{K}(N),X_{K,0}(N)$ and $X_{K,1}(N)$ over $K$, which are geometric analogues of the complex modular curves $X(N),X_{0}(N)$ and $X_{1}(N)$. 

\begin{question}
Can one compute the space $\M_{k}(N;K)$ of geometric modular forms of weight $k$ and level $N$ over an algebraically closed field $K$ of characteristic $p > 0$? In particular, can this be done when the corresponding stacky modular curve has wild ramification? 
\end{question}

For example, in \cite{kat}, Katz observes that if $A$ is the Hasse invariant and $\Delta$ is the modular discriminant, then $A\Delta$ is a cusp form of weight $13$ over $\F_{2}$ (resp.~of weight $14$ over $\F_{3}$) but this cannot be the reduction mod $2$ (resp.~mod $3$) of a modular form over $\Z$. More generally, we expect the orbifold structures of $X_{K}(N),X_{K,0}(N)$ and $X_{K,1}(N)$ to play a role in the determination of $\M_{k}(N;K)$. 

Now suppose $\char K = p > 0$. Then the covers of curves 
$$
X_{K}(N)\rightarrow X_{K,1}(N)\rightarrow X_{K,0}(N)\rightarrow X_{K}(1) = \P_{K}^{1}
$$
may have points with wild ramification. In fact, one such example is pointed out in \cite[Rmk. 5.3.11]{vzb}, which in turn cites an article \cite{bcg} that describes the ramification of the full cover $X_{\F_{p}}(\ell)\rightarrow X_{\F_{p}}(1) = \P_{\F_{p}}^{1}$ for primes $\ell\not = p$. These covers can have nonabelian inertia groups -- something that only happens in finite characteristic -- and they can then be decomposed and studied more carefully in the tower $X_{\F_{p}}(\ell)\rightarrow X_{\F_{p},1}(\ell)\rightarrow X_{\F_{p},0}(\ell)\rightarrow \P_{\F_{p}}^{1}$. The author plans to bring the results in Section~\ref{sec:canrings} to bear on this problem in the future.

\subsection{Artin--Schreier--Witt Theory for Stacky Curves}

In another direction, Theorems~\ref{thm:pcoverfactorsthroughASrootstack} and~\ref{thm:locallyASrootstack} essentially give a complete solution to the problem of classifying stacky curves in characteristic $p > 0$ having `at worst' $\Z/p\Z$ automorphism groups. To handle higher order cyclic automorphism groups, one needs the full power of {\it Artin--Schreier--Witt theory}, which we briefly review here. Suppose $k$ is a perfect field of characteristic $p > 0$ and $L/k$ is a Galois extension with group $G = \Z/p^{n}\Z$. When $n = 1$, we know by Artin--Schreier theory that such extensions are all of the form $L = k[x]/(x^{p} - x - a)$ for some $a\in k$, with different isomorphism classes coming from different valuations $v(a)$. 

For a general cyclic extension of order $p^{n}$, Artin--Schreier--Witt theory and the arithmetic of Witt vectors encode the automorphism data of the extension in a systematic way. The basic theory can be found in various places, including \cite[p.330]{lang}, but here is a summary. For a commutative ring $A$ of characteristic $p > 0$, the ring of Witt vectors $\W(A)$ is an object that ``lifts'' the arithmetic of $A$ to characteristic $0$. For example, $\W(\F_{p}) \cong \Z_{p}$, the ring of $p$-adic integers. In particular, there is a Frobenius operator $F : \W(A)\rightarrow\W(A)$ lifting the $p$th power operation $a\mapsto a^{p}$ to characteristic $0$. This allows one to classify cyclic $p$th power extensions of a field of characteristic $p$ in an analogous way to Kummer theory (for tame cyclic extensions) and Artin--Schreier theory (for $\Z/p\Z$-extensions). Specifically, any $\Z/p^{n}\Z$-extension can be obtained by adjoining roots of an equation in $\W(A)$ of the form $F\vec{x} - \vec{x} = \vec{a}$. 

Let $n\geq 1$ and let $\W_{n} = \W_{n}(k)$ denote the ring of length $n$ Witt vectors. In order to study the geometry of cyclic $p^{n}$th-order covers of curves in characteristic $p$, Garuti \cite{gar} introduced the following geometric version of Artin--Schreier--Witt theory. Define schemes $(\overline{\W}_{n},\orb_{\overline{\W}_{n}}(1))$ inductively by 
\begin{align*}
  (\overline{\W}_{1},\orb_{\overline{\W}_{1}}(1)) &= (\P^{1},\orb_{\P^{1}}(1))\\
  \text{and}\quad (\overline{\W}_{n},\orb_{\overline{\W}_{n}}(1)) &= (\P(\orb_{\overline{\W}_{n - 1}}\oplus\orb_{\overline{\W}_{n - 1}}(p)),\orb_{\P}(1)) \quad\text{for } n\geq 2,
\end{align*}
where $\P(E)$ denotes the projective bundle of a vector bundle $E$ and $\orb_{\P}(1)$ is the tautological bundle of the projective bundle in that step. These schemes assemble into a tower $\cdots\rightarrow\overline{\W}_{3}\rightarrow\overline{\W}_{2}\rightarrow\overline{\W}_{1} = \P^{1}$ which contains a wealth of information about cyclic covers. For us, the most important features of these $\overline{\W}_{n}$ are: 
\begin{itemize}
  \item Each $\overline{\W}_{n}$ is a projective scheme. In fact, $\cdots\rightarrow\overline{\W}_{3}\rightarrow\overline{\W}_{2}\rightarrow\overline{\W}_{1}$ is an example of a {\it Bott tower}, so each member is a toric variety. (For more on this topic, see \cite{cs}.) 
  \item For each $n$, there is an open immersion $j_{n} : \W_{n}\hookrightarrow\overline{\W}_{n}$. 
  \item The action of $\W_{n}$ on itself by Witt-vector translation extends to an action of $\W_{n}$ on $\overline{\W}_{n}$. Garuti calls this an {\it equivariant compactification} of the ring $\W_{n}$. 
  \item The higher ramification data of a cyclic $\Z/p^{n}\Z$-cover of curves can be determined explicitly by intersection theory in $\overline{\W}_{n}$ involving the boundary divisor $B_{n} := \overline{\W}_{n}\smallsetminus j_{n}(\W_{n})$. 
\end{itemize}

At the bottom of the tower, we can view $\overline{\W}_{1} = \P^{1}$ as the starting place for our theory of Artin--Schreier root stacks: we replaced $\P^{1}$ with the weighted projective line $\P(1,m)$, extended the map $\wp = Fx - x$ to the compactification $\P(1,m)$ and pulled back along the induced map $[\P(1,m)/\G_{a}]\rightarrow [\P(1,m)/\G_{a}]$ to take an Artin--Schreier root. Set $\overline{\W}_{1}(1,m) := \P(1,m)$. We can regard this as a stacky compactification of the ring of length $1$ Witt vectors $\W_{1} = \G_{a}$. 

The key insight for generalizing this is to use the fact (Lemma~\ref{lem:weightedprojrootstack}) that $\P(1,m)$ is itself a root stack over $\P^{1}$. Pulling back the root stack structure along the tower $\cdots\rightarrow\overline{\W}_{3}\rightarrow\overline{\W}_{2}\rightarrow\overline{\W}_{1} = \P^{1}$ defines $\overline{\W}_{n}(1,m,1,\ldots,1)$ for each $n > 1$. Each of these is a root stack over $\overline{\W}_{n}$ with stacky structure at the preimage of the point $\infty\in\P(1,m)$. Further weights can be introduced by iterated root stacks, and ultimately this leads to the definition of the {\it compactifed Witt stack} $\overline{\W}_{n}(1,m_{1},\ldots,m_{n})$ for any sequence of positive integers $(m_{1},\ldots,m_{n})$. In a sequel to the present article, the author plans to describe these stacks explicitly and use their properties to study $p$th-power order stack structures, as a generalization of the program in Section~\ref{sec:ASrootstacks}.



\end{document}